\documentclass[final,leqno,showlabe]{siamltex}
\usepackage{amsmath}
\usepackage{graphicx}
\usepackage{graphicx,epstopdf}
\usepackage{wrapfig}
\usepackage[caption=false]{subfig}
\usepackage{amssymb}
\usepackage{cases}

\usepackage{comment}

\usepackage{enumerate}
\usepackage{stmaryrd}
\usepackage{mathrsfs}
\usepackage{multirow}
\renewcommand{\theequation}{\arabic{section}.\arabic{equation}}
\usepackage[usenames]{color}
\usepackage{indentfirst}

\usepackage{array} 

\usepackage[bookmarksopen,bookmarksopenlevel=0,bookmarksdepth=2]{hyperref}
\usepackage{cleveref}

\usepackage{doi}

\numberwithin{equation}{section}

\SetSymbolFont{stmry}{bold}{U}{stmry}{m}{n}

\newtheorem{remark}{Remark}[section]

\newcommand{\abs}[1]{\left\lvert#1\right\rvert}

\newcommand{\be}{\begin{equation}}
\newcommand{\ee}{\end{equation}}

\newcommand{\rd}{\mathrm{d}}

\renewcommand{\vec}[1]{\mathbf{#1}}
\renewcommand{\b}[1]{\boldsymbol{#1}}

\title{A symmetrized parametric finite element method for
 anisotropic surface diffusion of closed curves}

\author{Weizhu Bao\thanks{Department of Mathematics, National University of Singapore, Singapore, 119076 (matbaowz@nus.edu.sg). This author's research was supported by the Ministry of Education of Singapore grant MOE2019-T2-1-063 (R-146-000-296-112).}
\and Wei Jiang\thanks{School of Mathematics and Statistics {\rm\&} Hubei Key Laboratory of Computational Science, Wuhan University, Wuhan, 430072, P. R. China (jiangwei1007@whu.edu.cn). This author's research was supported by the National Natural Science Foundation of China Nos. 12271414 and 11871384, and the National Key Research and Development Program of China No. 2020YFA0714201.}
\and Yifei Li\thanks{Department of Mathematics, National University of
Singapore, Singapore, 119076 (e0444158@u.nus.edu).}
}

\date{}

\begin{document}
\maketitle


\begin{abstract}
We deal with a long-standing problem about how to design an energy-stable numerical scheme for solving the motion of a closed curve under {\sl anisotropic surface diffusion} with a general anisotropic surface energy $\gamma(\boldsymbol{n})$ in two dimensions,
where $\boldsymbol{n}$ is the outward unit normal vector.
By introducing a novel surface energy matrix $\boldsymbol{Z}_k(\boldsymbol{n})$ which depends on the Cahn-Hoffman $\boldsymbol{\xi}$-vector and a stabilizing function $k(\boldsymbol{n})$, we first reformulate the equation into a conservative form, and derive a new symmetrized variational formulation for anisotropic surface diffusion with weakly or strongly anisotropic surface energies. Then, a semi-discretization in space for the variational formulation is proposed, and its area conservation and energy dissipation properties are proved. The semi-discretization is further discretized in time by an implicit structural-preserving scheme (SP-PFEM) which can rigorously preserves the enclosed area in the fully-discrete level. Furthermore, we prove that the SP-PFEM is unconditionally energy-stable for almost any anisotropic surface energy $\gamma(\boldsymbol{n})$ under a simple and mild condition on $\gamma(\boldsymbol{n})$. For several commonly-used anisotropic surface energies, we construct $\boldsymbol{Z}_k(\boldsymbol{n})$ explicitly. Finally, extensive numerical results are reported to demonstrate the high performance of the proposed scheme.
\end{abstract}




\pagestyle{myheadings} \markboth{W.~Bao, W.~Jiang and Y.~Li}
{Symmetrized PFEM for anisotropic surface diffusion}

\section{Introduction}

Anisotropic surface diffusion is an important and common process at material surfaces/interfaces in solids due to different surface lattice orientations. The lattice orientational difference leads to anisotropic surface energy in solid materials. It thus generates an anisotropic evolution process for a solid material. Recently, anisotropic surface diffusion has been regarded as an important kinetic process in surface phase formation, epitaxial growth, heterogeneous catalysis, and many other areas in surface/materials science \cite{Sutton95}. It has been witnessing significant and broader applications in materials science, computational geometry and solid-state physics, such as the evolution of voids in microelectronic circuits \cite{li1999numerical,Suo97}, microstructure evolution in solids \cite{cahn1991stability,Fonseca14}, the smoothing of discrete surfaces \cite{clarenz2000anisotropic}, and solid-state dewetting \cite{Thompson12,Ye10a,Jiang2012,wang2015sharp}.

As shown in Fig.~\ref{fig: illustartion figure}, for a closed curve $\Gamma$ in two dimensions (2D) associated with a given anisotropic surface energy
$\gamma(\boldsymbol{n})$, where $\boldsymbol{n}=(n_1,n_2)^T\in \mathbb{S}^1$ representing the unit outward normal vector, the motion by anisotropic surface diffusion of the curve is described by the following geometric evolution equation \cite{Mullins57, Cahn94}
\begin{equation}
\label{eq:sf2d}
V_n=\partial_{ss}\mu,
\end{equation}
where $V_n$ is the normal velocity, $s$ is the arclength parameter of $\Gamma$, and $\mu:=\mu(s)$ is the chemical potential (or weighted curvature denoted as $\kappa_\gamma:=\kappa_\gamma(s)$ in the literature \cite{taylor1992ii})
generated from the energy functional $W(\Gamma):=\int_\Gamma \gamma(\boldsymbol{n})ds$ via the thermodynamic variation \cite{Cahn94,Bao17}.
It is well-known that the anisotropic surface diffusion has the following two essential geometric properties: (i) the area of the region enclosed by the curve is conserved, and (ii) the free energy (or weighted length) $W(\Gamma)$ of the curve decreases in time \cite{wang2015sharp,Bao17,Kovacs21}. More precisely, the motion by anisotropic surface diffusion is the $H^{-1}$-gradient flow of the free energy (or weighted length) functional $W(\Gamma)$ \cite{taylor1994linking,Mayer01}.
\begin{figure}[hbtp!]
\centering
\includegraphics[width=0.6\textwidth]{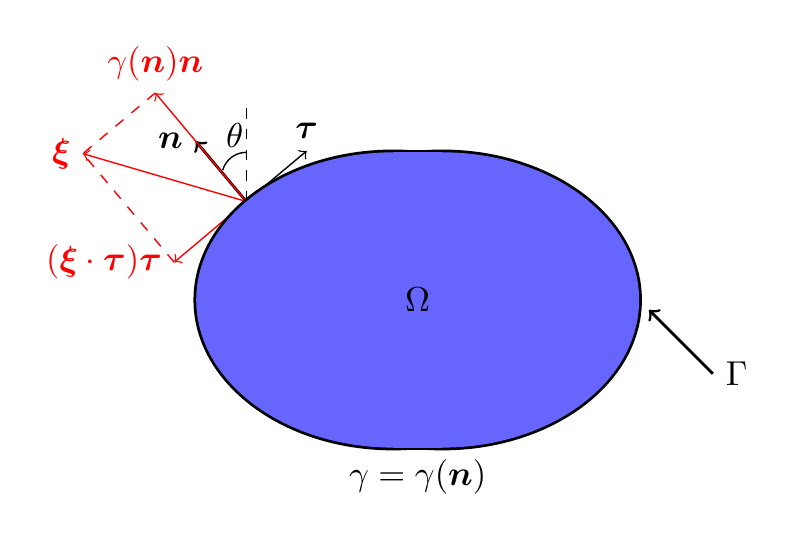}
\caption{An illustration of a closed curve $\Gamma$ in $\mathbb{R}^2$ under anisotropic surface diffusion with an anisotropic surface energy $\gamma(\boldsymbol{n})$, where $\boldsymbol{n}$ is the outward unit normal vector, $\boldsymbol{\tau}$ is the unit tangential vector,  $\boldsymbol{\xi}$ is the Cahn-Hoffman vector in \eqref{The xi-vector}, and $\theta$ is the angle between $\boldsymbol{n}$ and $y$-axis such that $\boldsymbol{n}=(-\sin\theta,\cos\theta)^T$ with $\theta\in[-\pi,\pi]$.}
\label{fig: illustartion figure}
\end{figure}

Let $\gamma(\boldsymbol{p}):\ \mathbb{R}^2\to {\mathbb R}$ be a homogeneous extension of the anisotropic surface energy
$\gamma(\boldsymbol{n}):\ \mathbb{S}^1\to {\mathbb R}^+$ satisfying:
(i) $\gamma(\boldsymbol{p})|_{\boldsymbol{p}=\boldsymbol{n}}
=\gamma(\boldsymbol{n})$ for $\boldsymbol{n}\in \mathbb{S}^1$, and  (ii) $\gamma(c\boldsymbol{p})=c\gamma(\boldsymbol{p})$ for $c>0$ and $\boldsymbol{p}\in \mathbb{R}^2$. A typical homogeneous extension is widely used in the literature as \cite{jiang2019sharp,deckelnick2005computation}\color{black}
\begin{equation}\label{gamma p}
\gamma(\boldsymbol{p}):=\left\{\begin{aligned}&|\boldsymbol{p}|\gamma\left(\frac{\boldsymbol{p}}
{|\boldsymbol{p}|}\right),\qquad \forall \boldsymbol{p}=(p_1,p_2)^T\in \mathbb{R}^2_*:=\mathbb{R}^2\setminus \{\boldsymbol{0}\},\\
&0, \qquad\qquad\qquad\quad \boldsymbol{p}=\boldsymbol{0},\end{aligned}\right.
\end{equation}\color{black}
where $|\boldsymbol{p}|=\sqrt{p_1^2+p_2^2}$.
Then the Cahn-Hoffman $\boldsymbol{\xi}$-vector introduced by Cahn and Hoffman is mathematically given by \cite{hoffman1972vector,wheeler1999cahn}
\begin{equation}\label{The xi-vector}
\boldsymbol{\xi}:=\boldsymbol{\xi}(\boldsymbol{n})=\nabla \gamma(\boldsymbol{p})\big|_{\boldsymbol{p}=\boldsymbol{n}}
=\gamma(\boldsymbol{n})\boldsymbol{n}+(\boldsymbol{\xi}\cdot\boldsymbol{\tau})
\boldsymbol{\tau}, \qquad \forall\boldsymbol{n}\in \mathbb{S}^1,
\end{equation}
where $\boldsymbol{\tau}=\boldsymbol{n}^\perp$ is the unit tangential vector
with the notation $^\perp$ denoting clockwise rotation by $\frac{\pi}{2}$
(cf. Fig. \ref{fig: illustartion figure}). Furthermore, the chemical potential $\mu$ (or weighted curvature) and the Hessian matrix $\bf{H}_{\gamma}(\boldsymbol{n})$ are defined as \cite{jiang2019sharp}
\begin{equation}\label{def of mu, original}
\mu:=-\boldsymbol{n}\cdot \partial_s \boldsymbol{\xi}^{\perp},\qquad \bf{H}_{\gamma}(\boldsymbol{n}):=
\nabla\nabla\gamma(\boldsymbol{p})\big|_{\boldsymbol{p}=\boldsymbol{n}}, \qquad \forall\boldsymbol{n}\in \mathbb{S}^1.
\end{equation}
We remark here that ${\bf H}_{\gamma}(\boldsymbol{n})\boldsymbol{n}
=\boldsymbol{0}$
and thus $0$ is an eigenvalue of $\bf{H}_{\gamma}(\boldsymbol{n})$ and $\boldsymbol{n}$ is a corresponding eigenvector. We denote the other eigenvalue of $\bf{H}_{\gamma}(\boldsymbol{n})$  as $\lambda({\boldsymbol{n}})\in{\mathbb R}$. \color{black}The Frank diagram $\mathcal{F}$ of $\gamma(\boldsymbol{n})$ is defined as $\mathcal{F}:=\{\boldsymbol{p}\in \mathbb{R}^2|\,\gamma(\boldsymbol{p})\leq 1\}$, i.e., $1/\gamma$ plot (see Page 190 in \cite{deckelnick2005computation}). \color{black}

Let $\Gamma:= \Gamma (t)$ be parameterized by $\vec X:=\vec X(s,t)=(x(s,t),y(s,t))^T\in{\mathbb R}^2$ with $t$ representing the time and $s$ denoting the arclength parametrization of $\Gamma$
(cf. Fig. \ref{fig: illustartion figure}), then via the Cahn-Hoffman $\boldsymbol{\xi}$-vector in \eqref{The xi-vector}, the anisotropic surface diffusion equation \eqref{eq:sf2d} for $\Gamma$ is described as follows \cite{jiang2019sharp}:
\begin{subequations}
\label{eqn:original formulation gamma(n) 1}
\begin{numcases}{}
\label{eqn:original gamma(n) aniso eq1 1}
\partial_t\boldsymbol{X} =\partial_{ss}\mu\, \boldsymbol{n}, \qquad 0<s<L(t), \quad t>0,  \\
\label{eqn:original gamma(n) aniso eq2 1}
\mu = -\boldsymbol{n}\cdot\partial_s \boldsymbol{\xi}^{\perp}, \quad
\boldsymbol{\xi}=\nabla \gamma(\boldsymbol{p})\big|_{\boldsymbol{p}=\boldsymbol{n}},
\end{numcases}
\end{subequations}
where $L(t)=\int_{\Gamma(t)} ds$ is the length of $\Gamma(t)$, and
\begin{equation}\label{ntauform}
\boldsymbol{\tau}=\partial_{s} \boldsymbol{X}=\boldsymbol{n}^\perp, \qquad
\boldsymbol{n}=-\partial_{s}\boldsymbol{X}^\perp=-\boldsymbol{\tau}^\perp.
\end{equation}
The initial data for
\eqref{eqn:original formulation gamma(n) 1} is given as
\begin{equation}\label{init}
\vec X(s,0)=\vec X_0(s)=(x_0(s),y_0(s))^T, \qquad 0\le s\le L_0,
\end{equation}
where $L_0$ represents the length of the initial curve $\Gamma_0=\Gamma(0)$.

When $\gamma(\boldsymbol{n})\equiv 1$,
it is named as {\sl isotropic} surface energy; in the isotropic case,
$\gamma(\boldsymbol{p})=|\boldsymbol{p}|$ in \eqref{gamma p},
$\boldsymbol{\xi}=\boldsymbol{n}$ in \eqref{The xi-vector}, and
$\mu=\kappa$ and ${\bf H}_{\gamma}(\boldsymbol{n})\equiv I_2-\boldsymbol{n}\boldsymbol{n}^T$ in \eqref{def of mu, original} with $\kappa$ the curvature and $I_2$ the $2\times 2$ identity matrix and $\lambda(\boldsymbol{n})\equiv 1$,
and thus \eqref{eqn:original formulation gamma(n) 1} collapses to the (isotropic) surface diffusion \cite{barrett2007parametric,Mullins57,Jiang,bao2020energy}.
In contrast, when $\gamma(\boldsymbol{n})$ is not a constant,
it is named as {\sl anisotropic} surface energy; and in the anisotropic case,
when $\boldsymbol{\tau}^T{\bf H}_{\gamma}(\boldsymbol{n})\boldsymbol{\tau}\ge
0$ for all $\boldsymbol{n}\in\mathbb{S}^1$ with $\boldsymbol{\tau}=\boldsymbol{n}^\perp$ ($\Leftrightarrow \lambda({\boldsymbol{n}})\ge0$ for $\boldsymbol{n}\in\mathbb{S}^1$
$\Leftrightarrow \tilde\gamma(
\theta):=\hat\gamma(\theta)+\hat\gamma^{\prime\prime}(\theta)\ge0$ for all $
\theta\in[-\pi,\pi]$ with $\hat\gamma(\theta):=\gamma(\boldsymbol{n})=\gamma(-\sin\theta,\cos\theta)$ $\Leftrightarrow$ the Frank diagram of $\gamma(\boldsymbol{n})$ is convex),
\color{black}  it is called as {\sl weakly anisotropic}; otherwise, when $\boldsymbol{\tau}^T{\bf H}_{\gamma}(\boldsymbol{n})\boldsymbol{\tau}$ changes sign for $\boldsymbol{n}\in\mathbb{S}^1$ ($\Leftrightarrow \lambda({\boldsymbol{n}})$ changes sign for $\boldsymbol{n}\in\mathbb{S}^1$
 $\Leftrightarrow \tilde\gamma(
\theta)$ changes sign for $\theta\in[-\pi,\pi]$  $\Leftrightarrow$ the Frank diagram of $\gamma(\boldsymbol{n})$ is not convex), it is called as {\sl strongly anisotropic}.

\color{black}
 Some commonly-used anisotropic surface energies $\gamma(\boldsymbol{n})$ are summarized as below:

(i) the \color{black} Riemannian-like \color{black} metric (also called as BGN) anisotropic surface energy \cite{barrett2008numerical,Barrett2020}
\begin{equation}\label{ellipsoidal}
	\gamma(\boldsymbol{n})=\sum_{l=1}^L\gamma_l(\boldsymbol{n})=
\sum_{l=1}^L\sqrt{\boldsymbol{n}^T\boldsymbol{G}_l\boldsymbol{n}},\qquad \forall\boldsymbol{n}\in \mathbb{S}^1,
\end{equation}
where $\boldsymbol{G}_l\in{\mathbb R}^{2\times 2}, l=1,\ldots,L$, are symmetric and positive definite matrices;

(ii) the $l^r$-norm metric anisotropic surface energy  \cite{deckelnick2005computation}
\begin{equation} \label{lrnormase}
  \gamma(\boldsymbol{n})=\left\|\boldsymbol{n}\right\|_{l^r}=
  \left(|n_1|^r+|n_2|^r\right)^{\frac{1}{r}},\qquad \forall\boldsymbol{n}=(n_1,n_2)^T\in \mathbb{S}^1,
\end{equation}
where $1< r<\infty$;

(iii) the $m$-fold anisotropic surface energy \cite{bao2017parametric}
\begin{align}\label{kfold}
  \gamma(\boldsymbol{n})=1+\beta \cos(m(\theta-\theta_0)), \qquad \forall\boldsymbol{n}=(n_1,n_2)^T=(-\sin\theta,\cos\theta)^T\in \mathbb{S}^1,
\end{align}
where $m=2,3,4,6$, $\theta_0\in [0,\pi]$ is a phase shift angle, and $\beta\ge0$ controls the degree of the anisotropy;

(iv) the regularized $l^1$-norm metric anisotropic surface energy which can be viewed as
a regularization for the non-smooth surface energy $\gamma(\boldsymbol{n})=|n_1|+|n_2|$
 \cite{barrett2008numerical,Barrett11}
\begin{align} \label{l1normase}
  \gamma(\boldsymbol{n})=\sqrt{n_1^2+\varepsilon^2n_2^2}
  +\sqrt{\varepsilon^2n_1^2+n_2^2},
  \qquad \forall\boldsymbol{n}=(n_1,n_2)^T\in \mathbb{S}^1,
\end{align}
where $0<\varepsilon\ll 1$ is a small `artificial' regularization parameter. \color{black}This regularization can be treated as a special case of \eqref{ellipsoidal}.\color{black}

\noindent For the convenience of readers, we list $\gamma(\boldsymbol{p})$,
$\boldsymbol{\xi}({\boldsymbol{n}})$, $\lambda({\boldsymbol{n}})$ and ${\bf H}_{\gamma}(\boldsymbol{n})$ of the above surface energies in Appendix A.


Different numerical methods have been proposed for the isotropic/anisotropic surface diffusion, such as the marker-particle method \cite{wong2000periodic,du2010tangent}, the finite element method via graph representation \cite{bansch2004surface,deckelnick2005computation,deckelnick2005fully}, the $\theta$-$L$ formulation method \cite{huang2021theta}, the discontinuous Galerkin finite element method \cite{xu2009local}, and the parametric finite element method (PFEM) \cite{barrett2007parametric,barrett2008variational,Hausser07,bao2017parametric,jiang2019sharp,li2020energy}. Among these methods, the PFEM performs the best in terms of accuracy and efficiency as well as mesh quality in practical computations via reformulating
\eqref{eqn:original formulation gamma(n) 1} as
\cite{jiang2019sharp}
\begin{subequations}
\label{eqn:pfem formulation gamma(n) 1}
\begin{numcases}{}
\label{eqn:pfem gamma(n) aniso eq1 1}
\boldsymbol{n}\cdot\partial_t\boldsymbol{X} =\partial_{ss}\mu, \qquad 0<s<L(t), \quad t>0,  \\
\label{eqn:pfem gamma(n) aniso eq2 1}
\mu\, \boldsymbol{n}= -\partial_s \boldsymbol{\xi}^{\perp}, \quad
\boldsymbol{\xi}=\nabla \gamma(\boldsymbol{p})\big|_{\boldsymbol{p}=\boldsymbol{n}}.
\end{numcases}
\end{subequations}
When $\gamma(\boldsymbol{n})\equiv 1$ (i.e., isotropic surface energy),
noting $\mu=\kappa$ and $\vec n= -\partial_s\vec X^\perp$, then
\eqref{eqn:pfem gamma(n) aniso eq2 1} collapses to $\kappa\,\boldsymbol{n}= -\partial_{ss}\boldsymbol{X}$. In this case, the PFEM is semi-implicit,
unconditionally energy-stable, and enjoys asymptotic equal mesh distribution \cite{barrett2007parametric} and thus there is no need to re-mesh during time evolution. Very recently, a structure-preserving PFEM (SP-PFEM) was proposed
for the surface diffusion \cite{bao2021structurepreserving,bao2022volume}. However, when the PFEM is extended directly to simulate anisotropic surface diffusion, many good properties are no longer preserved, especially for the unconditional energy stability, which can be preserved only for a very special
\color{black}Riemannian-like \color{black} metric anisotropic surface energy in \eqref{ellipsoidal} with a modified variational formulation \cite{barrett2008numerical}. Recently,
by reformulating \eqref{eqn:pfem gamma(n) aniso eq2 1} into a conservative form, an energy-stable PFEM
was designed for weakly anisotropic surface diffusion under a very strong condition on $\hat \gamma(\theta)=\gamma(-\sin\theta, \cos\theta)$ \cite{li2020energy}. To our best knowledge, it is still an open question to design an unconditionally energy-stable scheme for solving the anisotropic surface diffusion
\eqref{eqn:pfem formulation gamma(n) 1} with any form of $\gamma(\boldsymbol{n})$.

The objective of this paper is to propose an unconditionally energy-stable SP-PFEM for solving the anisotropic surface diffusion
\eqref{eqn:pfem formulation gamma(n) 1} with the surface energy $\gamma(\boldsymbol{n})$ satisfying a relatively mild condition as
\begin{equation}\label{engstabgmp}
\gamma(-\boldsymbol{n})=\gamma(\boldsymbol{n}), \quad \forall \boldsymbol{n}\in \mathbb{S}^1, \qquad \gamma(\boldsymbol{p})\in C^2(\mathbb{R}^2\setminus \{\boldsymbol{0}\}).
\end{equation}
We first reformulate \eqref{eqn:pfem gamma(n) aniso eq2 1} into a conservative and self-adjoint form by introducing a novel symmetric positive definite surface energy matrix $\boldsymbol{Z}_k(\boldsymbol{n})$ depending on the Cahn-Hoffman $\boldsymbol{\xi}$-vector and a stabilizing function $k(\boldsymbol{n})$, and then derive a new symmetrized variational formulation for the anisotropic surface diffusion
\eqref{eqn:pfem formulation gamma(n) 1}. The symmetrized variational formulation is first discretized in space by PFEM and then
discretized in time by an implicit SP-PFEM which preserves the area in the fully-discrete level. Under the  simple and mild condition
\eqref{engstabgmp} on $\gamma(\boldsymbol{n})$,
we rigorously prove that the SP-PFEM is energy dissipative and thus is unconditionally energy stable for almost all anisotropic surface energy $\gamma(\boldsymbol{n})$ arising in practical applications, including both weakly and strongly anisotropic surface energies.

The remainder of this paper is structured as follows: In section 2, we first introduce the surface energy matrix $\boldsymbol{Z}_k(\boldsymbol{n})$, propose a new symmetrized variational formulation and show its area conservation and energy dissipation. In section 3, we present a semi-discretization
in space by PFEM and a full-discretization by an implicit SP-PFEM for the weak formulation. In section 4, we prove the unconditional energy-stability of SP-PFEM under
the condition \eqref{engstabgmp} on $\gamma(\boldsymbol{n})$. In section 5, numerical results are given to demonstrate
the high performance of the proposed SP-PFEM. Finally, we draw some conclusions in section 6.

\section{A new symmetrized variational formulation and its properties}

In this section, we present a new conservative and self-adjoint
formulation of \eqref{eqn:pfem gamma(n) aniso eq2 1} and a new symmetrized variational formulation of \eqref{eqn:pfem formulation gamma(n) 1}, and
prove the area conservation and energy dissipation of the new formulation.

\subsection{A symmetric positive definite surface energy matrix}
Introducing a symmetric surface energy matrix $\boldsymbol{Z}_k(\boldsymbol{n})$ as
  \begin{equation}\label{def of the energy matrix Z_k(n)}
  	\boldsymbol{Z}_k(\boldsymbol{n})=\gamma(\boldsymbol{n})I_2-\boldsymbol{n} \boldsymbol{\xi}(\boldsymbol{n})^T-\boldsymbol{\xi}(\boldsymbol{n}) \boldsymbol{n}^T+k(\boldsymbol{n})\boldsymbol{n}\boldsymbol{n}^T,
  \qquad \forall\boldsymbol{n}\in \mathbb{S}^1,
  \end{equation}
where $k(\boldsymbol{n}):\ \mathbb{S}^1\to {\mathbb R}^+$ is a
stabilizing function to be determined later, then we have

\begin{lemma}[symmetric and conservative form] With the symmetric surface energy matrix $\boldsymbol{Z}_k(\boldsymbol{n})$ in \eqref{def of the energy matrix Z_k(n)},
the anisotropic surface diffusion \eqref{eqn:pfem formulation gamma(n) 1}
can be reformulated as
\begin{subequations}
\label{eqn:pfemm formulation gamma(n)}
\begin{numcases}{}
\label{eqn:pfemm gamma(n) aniso eq1}
\boldsymbol{n}\cdot \partial_t\boldsymbol{X} =\partial_{ss}\mu,   \\
\label{eqn:pfemm gamma(n) aniso eq2}
\mu\,\boldsymbol{n}=-\partial_{s}(\boldsymbol{Z}_k(\boldsymbol{n})\partial_s\boldsymbol{X}).
\end{numcases}
\end{subequations}
\end{lemma}

\begin{proof} From \eqref{The xi-vector}, noting \eqref{ntauform}, we get
\begin{equation}\label{xiperp23}
\boldsymbol{\xi}^{\perp}=\gamma(\boldsymbol{n})\boldsymbol{n}^\perp+
(\boldsymbol{\xi}\cdot\boldsymbol{\tau})\boldsymbol{\tau}^\perp
=\gamma(\boldsymbol{n})\boldsymbol{\tau}-
(\boldsymbol{\xi}\cdot\boldsymbol{\tau})\boldsymbol{n}.
\end{equation}
From \eqref{def of the energy matrix Z_k(n)}, noticing \eqref{ntauform} and
\eqref{xiperp23}, and using $\boldsymbol{n}\cdot\boldsymbol{\tau}=0$,
we get
\begin{align}\label{reformulation of Zk, auxillary eq 2}
\boldsymbol{Z}_k(\boldsymbol{n})\partial_s\boldsymbol{X}&=
\boldsymbol{Z}_k(\boldsymbol{n})\boldsymbol{\tau}=
(\gamma(\boldsymbol{n})I_2
-\boldsymbol{n}\boldsymbol{\xi}^T-\boldsymbol{\xi}\boldsymbol{n}^T
+k(\boldsymbol{n})\boldsymbol{n}
\boldsymbol{n}^T)\boldsymbol{\tau}\nonumber\\
&=\gamma(\boldsymbol{n})\boldsymbol{\tau}-
(\boldsymbol{\xi}\cdot\boldsymbol{\tau})\boldsymbol{n}+
(\boldsymbol{n}\cdot\boldsymbol{\tau}) \left(k(\boldsymbol{n})\boldsymbol{n}-\boldsymbol{\xi}\right)=
\boldsymbol{\xi}^{\perp}.
\end{align}
Plugging \eqref{reformulation of Zk, auxillary eq 2} into
\eqref{eqn:pfem formulation gamma(n) 1}, we obtain
\eqref{eqn:pfemm formulation gamma(n)} immediately.
\end{proof}


\begin{remark}
When $\gamma(\boldsymbol{n})\equiv 1$ and by taking $k(\boldsymbol{n})\equiv 2$ in \eqref{def of the energy matrix Z_k(n)}, we have $\mu=\kappa$ and $\boldsymbol{\xi}=\boldsymbol{n}$, and thus $\boldsymbol{Z}_k(\boldsymbol{n})\equiv I_2$.  Then \eqref{eqn:pfemm formulation gamma(n)} collapses to the standard formulation by PFEM for surface diffusion \cite{barrett2007parametric}. Similarly, when $\gamma(\boldsymbol{n})$ is chosen as the \color{black}Riemannian-like \color{black} metric anisotropic surface energy \eqref{ellipsoidal}, by taking $k(\boldsymbol{n})=\sum\limits_{l=1}^L\gamma_l(\boldsymbol{n})^{-1}\,\text{Tr}(\boldsymbol{G}_l)$ with $\text{Tr}(\boldsymbol{G}_l)$ denoting the trace of $G_l$, then
\eqref{eqn:pfemm formulation gamma(n)} collapses to the formulation used in \cite{barrett2008numerical}.
A similar formulation but without the symmetrizing term $-\boldsymbol{\xi}(\boldsymbol{n})\boldsymbol{n}^T$ and the stabilizing term $k(\boldsymbol{n})$
can also be found in ~\cite[(1.18)]{barrett2008variational}.
\end{remark}

\subsection{The variational formulation}
Let $\mathbb{T}=\mathbb{R}\slash\mathbb{Z}=[0,1]$ be the periodic unit interval and we parameterize the evolution curves $\Gamma(t)$ as
\begin{equation}
\Gamma(t):=\boldsymbol{X}(\color{black}\mathbb{T}\color{black},t),\,\, \boldsymbol{X}(\rho,t):=(x(\rho,t),~y(\rho,t))^T:\; \mathbb{T}\times \mathbb{R}^+\;\rightarrow \;\mathbb{R}^2.
\end{equation}
The arclength parameter $s$ is computed by $s(\rho,t)=\int_0^\rho \abs{\partial_q\vec{X}(q,t)}\,\rd q$ with its derivative $\partial_\rho s=\abs{\partial_\rho\vec{X}}$.
By the introduced time-independent variable $\rho$, the evolving curve $\Gamma(t)$ can then be parameterized over a fixed domain $\rho\in\mathbb{T}=[0,1]$. We do not distinguish the two parameterization $\vec X(\rho,t)$ and $\vec X(s,t)$ for $\Gamma(t)$ if there is no ambiguity.
We also introduce the usual Sobolev space as
\begin{equation}
L^2(\mathbb{T})=\left\{u: \mathbb{T}\rightarrow \mathbb{R} \ |\   \int_{\mathbb{T}} |u(\rho)|^2 \, d\rho <+\infty \right\},
\end{equation}
equipped with the weighted $L^2$-inner product with respect to the closed curve $\Gamma(t)$
\be\label{inner product torus}
\big(u,v\big)_{\Gamma(t)}:=\int_{\Gamma(t)}u(s)\,v(s)ds=\int_{\mathbb{T}}u(\rho) v(\rho) \partial_\rho s(\rho,t)\,d\rho,\quad \forall\;u,v\in L^2(\mathbb{T}),
\ee\color{black}
which can be easily extended to $[L^2(\mathbb{T})]^2$. Here, we always assume that $\partial_\rho s(\rho,t)$ is bounded for all $t$. Moreover, the Sobolev space $H^1(\mathbb{T})$ is given as
\begin{align}
H^1(\mathbb{T}):=\left\{u: \mathbb{T}\rightarrow \mathbb{R}, \;{\rm and}\;u\in L^2(\mathbb{T}),\; \partial_{\rho} u\in L^2(\mathbb{T})\right\}.
\end{align}

Multiplying a test function $\varphi(\rho)\in H^1(\mathbb{T})$ to \eqref{eqn:pfemm gamma(n) aniso eq1}, and then integrating over $\Gamma(t)$ and taking integration by parts,  we have
\begin{eqnarray}\label{varf1 continuous torus}
\Bigl(\boldsymbol{n}\cdot\partial_{t}\boldsymbol{X},
\varphi\Bigr)_{\Gamma(t)}=
\Bigl(\partial_{ss}\mu,\varphi\Bigr)_{\Gamma(t)}
=-\Bigl(\partial_{s}\mu,
\partial_s\varphi\Bigr)_{\Gamma(t)}.
\end{eqnarray}
Similarly, by multiplying a test function $\boldsymbol{\omega}=(\omega_1,\omega_2)^T\in [H^1(\mathbb{T})]^2$ to \eqref{eqn:pfemm gamma(n) aniso eq2}, we obtain
\begin{eqnarray}\label{varf2 continuous torus}
\Bigl(\mu\,\boldsymbol{n}, \boldsymbol{\omega}\Bigr)_{\Gamma(t)}
=\Bigl(-\partial_{s}(\boldsymbol{Z}_k(\boldsymbol{n})\partial_s\boldsymbol{X}), \boldsymbol{\omega}\Bigr)_{\Gamma(t)}
=\Bigl( \boldsymbol{Z}_k(\boldsymbol{n})\partial_s\boldsymbol{X},\partial_s
\boldsymbol{\omega}\Bigr)_{\Gamma(t)}.
\end{eqnarray}
By combining the two weak formulations \eqref{varf1 continuous torus} and \eqref{varf2 continuous torus}, we now get the novel symmetrized variational formulation for
the anisotropic surface diffusion \eqref{eqn:pfemm formulation gamma(n)}
(or \eqref{eqn:original formulation gamma(n) 1}) with the initial condition
\eqref{init}. More precisely, for a given initial curve $\Gamma_0:=\Gamma(0)=\boldsymbol{X}(\color{black}\mathbb{T}\color{black},0)$ with $\boldsymbol{X}(\rho,0)=\boldsymbol{X}_0(L_0\rho)\in [H^1(\mathbb{T})]^2$\color{black}, find the solution $\Gamma(t):=\color{black}\boldsymbol{X}(\mathbb{T},t),\color{black}\,\boldsymbol{X}(\cdot,t)\in [H^1(\mathbb{T})]^2$ and $\mu(\cdot,t)\in H^1(\mathbb{T})$ such that:
\begin{subequations}
\label{eqn:weak continuous torus}
\begin{align}
\label{eqn:weak1 continuous torus}
&\Bigl(\boldsymbol{n}\cdot\partial_{t}\boldsymbol{X},
\varphi\Bigr)_{\Gamma(t)}+\Bigl(\partial_{s}\mu,
\partial_s\varphi\Bigr)_{\Gamma(t)}=0,\qquad\forall \varphi\in H^1(\mathbb{T}),\\[0.5em]
\label{eqn:weak2 continuous torus}
&\Bigl(\mu,\boldsymbol{n}\cdot\boldsymbol{\omega}\Bigr)_{\Gamma(t)}-\Bigl( \boldsymbol{Z}_k(\boldsymbol{n})\partial_s \boldsymbol{X},\partial_s\boldsymbol{\omega}\Bigr)_{\Gamma(t)}= 0,\quad\forall\boldsymbol{\omega}\in [H^1(\mathbb{T})]^2.
\end{align}
\end{subequations}

\subsection{Area conservation and energy dissipation}
Let $A(t)$ denote the area (i.e., the  region $\Omega(t)$ enclosed by the curve $\Gamma(t)$) and $W_c(t)$ denote the free energy (or weighted length), which are defined as
\begin{equation}\label{AtWct}
A(t):=\int_{\Omega(t)}1\,d{\bf x}=\int_0^{L(t)}y(s,t)\partial_sx(s,t)\,ds,\ \ W_c(t):=\int_{\Gamma(t)}\gamma(\boldsymbol{n})\,ds, \ \ t\ge0.
\end{equation}
For the above variational problem
\eqref{eqn:weak continuous torus}, we have
\begin{proposition}[area conservation and energy dissipation] The area $A(t)$ of the solution $\Bigl(\boldsymbol{X}(\cdot,~t)$, $\mu(\cdot,~t)\Bigr)\in [H^1(\mathbb{T})]^2 \times  H^1(\mathbb{T})$ defined by the variational problem
\eqref{eqn:weak continuous torus} is conserved, and the energy $W_c(t)$ is dissipative, i.e.
\begin{equation}\label{mass and energy law weak form}
A(t)\equiv \int_0^{L_0}y_0(s)x_0^\prime(s)ds,
\quad W_c(t)\le W_c(t_1)\le \int_0^{L_0}\gamma(\boldsymbol{n})\,ds,\quad t\geq t_1 \ge0.
\end{equation}
\end{proposition}

\begin{proof} The proof of area conservation is
similar to the Proposition 2.1 in \cite{li2020energy}, thus we omit the details for brevity.

To prove the energy dissipation in \eqref{mass and energy law weak form},
taking the derivative of $W_c(t)$  with respect to $t$,
noting \eqref{The xi-vector}, \eqref{reformulation of Zk, auxillary eq 2},  \eqref{eqn:weak2 continuous torus} with $\boldsymbol{\omega}=\partial_t\boldsymbol{X}$, and
\eqref{eqn:weak1 continuous torus} with $\varphi=\mu$, and $\partial_t\boldsymbol{n}=(\boldsymbol{\tau}\cdot\partial_t \boldsymbol{n})\boldsymbol{\tau}
=-(\boldsymbol{n}\cdot\partial_s\partial_t\boldsymbol{X})
\boldsymbol{\tau}$, we have
\begin{eqnarray}\label{dGamma(t) continuous}
\dot{W}_c(t)&=&\frac{d}{dt}\int_0^{L(t)}\gamma(\boldsymbol{n})ds
  =\frac{d}{dt}\int_0^1\gamma(\boldsymbol{n})\partial_\rho s d\rho=\int_0^1(\gamma(\boldsymbol{n})\partial_t\partial_\rho s+\nabla \gamma(\boldsymbol{n})\cdot \partial_t\boldsymbol{n}\partial_\rho s)d\rho\nonumber\\
  &=&\int_0^1(\gamma(\boldsymbol{n})\boldsymbol{\tau}-(\boldsymbol{\xi}\cdot \boldsymbol{\tau})\boldsymbol{n})\cdot \partial_s\partial_t\boldsymbol{X}\partial_\rho s \,d\rho=\Bigl(\boldsymbol{Z}_k(\boldsymbol{n})\partial_s\boldsymbol{X},\partial_s\partial_t\boldsymbol{X}\Bigr)_{\Gamma(t)}\nonumber\\
  &=&-\Bigl(\partial_{s}\mu,
\partial_s\mu\Bigr)_{\Gamma(t)}\leq 0,  \nonumber
\end{eqnarray}
which implies the energy dissipation in
\eqref{mass and energy law weak form}.
\end{proof}

\section{PFEM discretizations and their properties}
In this section, we first discretize the variational problem
\eqref{eqn:weak continuous torus} in space by PFEM and show
its area conservation and energy dissipation.
Then we further discretize the semi-discretization in time
by a structure-preserving PFEM (SP-PFEM) which conserves area
in the fully-discrete level.

\subsection{A semi-discretization in space by PFEM and its properties}
Let $N>0$ be an integer, the mesh size $h=1/N$, the grid points $\rho_{j}=jh$ for $j=0,1,\ldots,N$, and the subintervals
$I_j=[\rho_{j-1},\rho_{j}]$ for $j=1,2,\ldots,N$. Then we can give
a uniform partition of the torus $\mathbb{T}$ by
$\mathbb{T}=[0,1]=\bigcup_{j=1}^{N}I_j$. Moreover, the finite element subspace of $H^1(\mathbb{T})$ is given by
 \begin{eqnarray*}\label{eqn:H1 semi torus}
\mathbb{K}^h=\mathbb{K}^h(\mathbb{T}):=\{u^h\in C(\mathbb{T})\ |\ u^h\mid_{I_{j}}\in \mathcal{P}_1,\ \forall j=1,2,\ldots,N\}, 
\end{eqnarray*}
where $\mathcal{P}_1$ stands for the space of polynomials of degree at most $1$.

Let the piecewise linear curve $\Gamma^h(t):=\color{black}\boldsymbol{X}^h(\mathbb{T},t),\,\color{black}\boldsymbol{X}^h(\cdot,t)=(x^h(\cdot,t),
y^h(\cdot,t))^T\in [\mathbb{K}^h]^2$ be the numerical approximation of $\Gamma(t):=\color{black}\boldsymbol{X}(\mathbb{T},t),\,\color{black}\boldsymbol{X}(\cdot,t)\in [H^1(\mathbb{T})]^2$ and the piecewise linear function
$\mu^h(\cdot,t)\in \mathbb{K}^h$ be the numerical approximation of $\mu(\cdot,t)\in H^1(\mathbb{T})$, where $(\vec X(\cdot,t), \mu(\cdot,t))\in [H^1(\mathbb{T})]^2\times H^1(\mathbb{T})$ is given by the variational problem
\eqref{eqn:weak continuous torus}. Then $\Gamma^h(t)$ is formed by ordered vectors $\{\boldsymbol{h}_j(t)\}_{j=1}^N$ and we assume that for $t\geq 0$, these vectors $\boldsymbol{h}_j(t)$ satisfy
\begin{eqnarray}\label{hjt987}
\quad h_{\rm min}(t):=\min_{1\le j\le N} |\boldsymbol{h}_j(t)|>0, \ \boldsymbol{h}_j(t):=\boldsymbol{X}^h(\rho_j,t)-\boldsymbol{X}^h(\rho_{j-1},t),\ \forall j,
\end{eqnarray}
where $|\boldsymbol{h}_j(t)|$ is the length of the vector $\boldsymbol{h}_j(t)$ for $j=1,2,\ldots,N$.

The outward unit normal vector $\boldsymbol{n}^h$, the unit tangential vector $\boldsymbol{\tau}^h$, and
the Cahn-Hoffman $\boldsymbol{\xi}$-vector $\boldsymbol{\xi}^h$
 of the curve $\Gamma^h(t)$ are constant vectors in the interior of each interval $I_j$ which can be computed by $\boldsymbol{h}_j(t)$ as
\begin{equation}
\boldsymbol{n}^h|_{I_j}=
-\frac{(\boldsymbol{h}_j)^\perp}{|\boldsymbol{h}_j|}:=\boldsymbol{n}^h_j,\quad \boldsymbol{\tau}^h|_{I_j}=\frac{\boldsymbol{h}_j}{|\boldsymbol{h}_j|}:
=\boldsymbol{\tau}^h_j,
\quad \boldsymbol{\xi}^h|_{I_j}=\boldsymbol{\xi}(\boldsymbol{n}^h_j):
=\boldsymbol{\xi}^h_j.
\label{eqn:semitannorm}
\end{equation}

Furthermore, for two scalar-/vector-valued functions $u$ and $v$ in $\mathbb{K}^h$ or $[\mathbb{K}^h]^2$ respectively, the mass lumped inner product $\big(\cdot,\cdot\big)_{\Gamma^h}^h$ over $\Gamma^h$ is defined as
\begin{equation}
\big(u,~v\big)_{\Gamma^h}^h:=\frac{1}{2}\sum_{j=1}^{N}|
\boldsymbol{h}_j|\,\Big[\big(u\cdot v\big)(\rho_j^-)+\big(u\cdot v\big)(\rho_{j-1}^+)\Big],
\label{eqn:semiproduct}
\end{equation}
where $u(\rho_j^\pm)=\lim\limits_{\rho\to \rho_j^\pm} u(\rho)$ for $0\le j\le N$.

Suppose $\Gamma^h(0):=\color{black}\boldsymbol{X}^h(\mathbb{T},0),\,\color{black}\boldsymbol{X}^h(\cdot,0)\in [\mathbb{K}^h]^2$ is the piecewise linear interpolation
of $\boldsymbol{X}_0(s)$ in \eqref{init}, where
$\boldsymbol{X}^h(\rho=\rho_j,0)=\boldsymbol{X}_0(s=s_j^0)$ with $s_j^0=L_0\rho_j$ for $j=0,1,\ldots,N$. Now we can state the following spatial semi-discretization  of the symmetrized variational formulation \eqref{eqn:weak continuous torus}: for a given initial curve $\Gamma^h(0):=\color{black}\boldsymbol{X}^h(\mathbb{T},0),\,\color{black}\boldsymbol{X}^h(\cdot,0)\in [\mathbb{K}^h]^2$, find the solution $\Gamma^h(t):=\color{black}\boldsymbol{X}^h(\mathbb{T},t),\,\color{black}\\\boldsymbol{X}^h(\cdot,t)=(x^h(\cdot,t),y^h(\cdot,t))^T\in
[\mathbb{K}^h]^2$ and $\mu^h(\cdot,t)\in \mathbb{K}^h$, such that
\begin{subequations}
\label{eqn:2dsemi torus}
\begin{align}
\label{eqn:2dsemi1 torus}
&\Bigl(\boldsymbol{n}^h\cdot\partial_{t}\boldsymbol{X}^h,
\varphi^h\Bigr)_{\Gamma^h}^h+\Bigl(\partial_{s}\mu^h,
\partial_s\varphi^h\Bigr)_{\Gamma^h}^h=0,\qquad\forall \varphi^h\in \mathbb{K}^h,\\
\label{eqn:2dsemi2 torus}
&\Bigl(\mu^h,\boldsymbol{n}^h\cdot\boldsymbol{\omega}^h\Bigr)_{\Gamma^h}^h-\Bigl( \boldsymbol{Z}_k(\boldsymbol{n}^h)\partial_s \boldsymbol{X}^h,\partial_s\boldsymbol{\omega}^h\Bigr)_{\Gamma^h}^h=0,
\quad\forall\boldsymbol{\omega}^h\in[\mathbb{K}^h]^2,
\end{align}
\end{subequations}
where
\begin{eqnarray*}
\boldsymbol{Z}_k(\boldsymbol{n}^h)&=&\gamma(\boldsymbol{n}^h)I_2-\boldsymbol{n}^h \boldsymbol{\xi}(\boldsymbol{n}^h)^T-\boldsymbol{\xi}(\boldsymbol{n}^h) (\boldsymbol{n}^h)^T+k(\boldsymbol{n}^h)\,\boldsymbol{n}^h(\boldsymbol{n}^h)^T\\
&=&\gamma(\boldsymbol{n}^h)I_2-\boldsymbol{n}^h (\boldsymbol{\xi}^h)^T-\boldsymbol{\xi}^h (\boldsymbol{n}^h)^T+k(\boldsymbol{n}^h)\,\boldsymbol{n}^h(\boldsymbol{n}^h)^T.
\end{eqnarray*}


Let $A^h(t)$ denote the area of the enclosed region of the piecewise linear closed curve $\Gamma^h(t)$, and $W^h_c(t)$ be its total
free energy, which are defined as
\begin{equation}\label{discrete area}
A^h(t)=\frac{1}{2}\sum_{j=1}^N[x_j^h(t)-x_{j-1}^h(t)][y_j^h(t)+y_{j-1}^h(t)],
\quad W^h_c(t)=\sum_{j=1}^N|\boldsymbol{h}_j(t)|\gamma(\boldsymbol{n}^h_j).
\end{equation}
\begin{remark}
Similar to the proof in \cite[Proposition 3.1]{li2020energy}, for the solution of the above semi-discretization \eqref{eqn:2dsemi torus},
we can easily prove the area conservation and energy dissipation during time evolution.
\end{remark}

\subsection{A structure-preserving PFEM}
Let $\tau>0$ be the time step size, and $t_m=m\tau$ be the discrete time levels for each $m\ge0$. Let $\Gamma^m\triangleq\Gamma^{h,m}=\color{black}\boldsymbol{X}^m(\mathbb{T}),\,\color{black}\boldsymbol{X}^m(\cdot)=
(x^m(\cdot),y^m(\cdot))^T\in [\mathbb{K}^h]^2$ is the numerical approximation of $\Gamma^h(t_m)=\color{black}\boldsymbol{X}^h(\mathbb{T},t_m),\,\color{black}\boldsymbol{X}^h(\cdot,t_m)\in [\mathbb{K}^h]^2$ and $\mu^m(\cdot)\in \mathbb{K}^h$ be the numerical  approximation of  $\mu^h(\cdot,t_m)\in \mathbb{K}^h$ for each $m\geq 0$, where $(\boldsymbol{X}^h(\cdot,t), \mu^h(\cdot,t))$ is
the solution of the semi-discretization \eqref{eqn:2dsemi torus}. Similarly, $\Gamma^{m}$ is formed by the ordered vectors
$\{\boldsymbol{h}^m_j\}_{j=1}^N$ defined by
\begin{equation}
\boldsymbol{h}^m_j
:=\boldsymbol{X}^m(\rho_j)-\boldsymbol{X}^m(\rho_{j-1}), \qquad  j=1,2,\ldots,N.
\end{equation}
Again, for each $m\geq 0$, the outward unit normal vector $\boldsymbol{n}^m$, the unit tangential vector $\boldsymbol{\tau}^m$, and the Cahn-Hoffman $\boldsymbol{\xi}$-vector $\boldsymbol{\xi}^m$ of the curve $\Gamma^{m}$ are constant vectors in the interior of each interval $I_j$ which can be computed as
\begin{equation}
\boldsymbol{n}^m|_{I_j}=-\frac{(\boldsymbol{h}_j^m)^\perp}
{|\boldsymbol{h}_j^m|}:=\boldsymbol{n}^m_j,\quad \boldsymbol{\tau}^m|_{I_j}=\frac{\boldsymbol{h}_j^m}
{|\boldsymbol{h}_j^m|}:=\boldsymbol{\tau}^m_j,
 \quad \boldsymbol{\xi}^m|_{I_j}=\boldsymbol{\xi}(\boldsymbol{n}^m_j)
:=\boldsymbol{\xi}_j^m.
\label{tangent and normal vector full case 1}
\end{equation}

Following the idea in~\cite{bao2021structurepreserving,jiang2021} to design
a SP-PFEM for surface diffusion, i.e., using the backward Euler method in time and the information of the curve at current time step and next time step to linearly interpolate the normal vector, a symmetrized SP-PFEM discretization of \eqref{eqn:2dsemi torus} is given as: for a given initial curve $\Gamma^{0}:=\color{black}\boldsymbol{X}^0(\mathbb{T})\color{black},\,\boldsymbol{X}^0(\cdot)\in [\mathbb{K}^h]^2$, for $m\geq 0$, find the curve $\Gamma^{m+1}:=\color{black}\boldsymbol{X}^{m+1}(\mathbb{T}),\,\color{black}\boldsymbol{X}^{m+1}(\cdot)
\in[\mathbb{K}^h]^2$ and  the chemical potential $\mu^{m+1}(\cdot)\in \mathbb{K}^h$, such that
\begin{subequations}
\label{eqn:aniso SP-PFEM}
\begin{align}
\label{eqn:aniso SP-PFEM eq1}
&\Bigl(\frac{\boldsymbol{X}^{m+1}-\boldsymbol{X}^m}{\tau}\cdot \boldsymbol{n}^{m+\frac{1}{2}},\varphi^h\Bigr)_{\Gamma^m}^h + \Bigl(\partial_s\mu^{m+1},\partial_s\varphi^h\Bigr)_{\Gamma^m}^h=0,\ \forall \varphi^h\in \mathbb{K}^h,\\
\label{eqn:aniso SP-PFEM eq2}
&\Bigl(\mu^{m+1},\boldsymbol{n}^{m+\frac{1}{2}}\cdot
\boldsymbol{\omega}^h\Bigr)_{\Gamma^m}^h-
\Bigl(\boldsymbol{Z}_k(\boldsymbol{n}^m)\partial_s \boldsymbol{X}^{m+1},\partial_s\boldsymbol{\omega}^h
\Bigr)_{\Gamma^m}^h=0,\ \forall\boldsymbol{\omega}^h
\in[\mathbb{K}^h]^2,
\end{align}
\end{subequations}
where $s$ is the arclength parameter of $\Gamma^{m}$, and $\boldsymbol{n}^{m+\frac{1}{2}}$ and $\boldsymbol{Z}_k(\boldsymbol{n}^m)$ are defined as
\begin{eqnarray}
\boldsymbol{n}^{m+\frac{1}{2}}&:=&-
  \frac{1}{2}\left(\partial_s\boldsymbol{X}^m+
  \partial_s\boldsymbol{X}^{m+1}\right)^{\perp}=
  -\frac{1}{2}\frac{1}{|\partial_\rho \boldsymbol{X}^m|}
  \left(\partial_\rho \boldsymbol{X}^m+\partial_\rho\boldsymbol{X}^{m+1}\right)^{\perp},\\
\boldsymbol{Z}_k(\boldsymbol{n}^m)&=&\gamma(\boldsymbol{n}^m)I_2-\boldsymbol{n}^m \boldsymbol{\xi}(\boldsymbol{n}^m)^T-\boldsymbol{\xi}(\boldsymbol{n}^m) (\boldsymbol{n}^m)^T+k(\boldsymbol{n}^m)\,\boldsymbol{n}^m(\boldsymbol{n}^m)^T
\nonumber\\
&=&\gamma(\boldsymbol{n}^m)I_2-\boldsymbol{n}^m (\boldsymbol{\xi}^m)^T-\boldsymbol{\xi}^m (\boldsymbol{n}^m)^T+k(\boldsymbol{n}^m)\,\boldsymbol{n}^m(\boldsymbol{n}^m)^T,
\end{eqnarray}
and for any scalar-/vector-valued function $f\in \mathbb{K}^h$ or $[\mathbb{K}^h]^2$ respectively, we compute its derivative with respect to the arclength parameter on $\Gamma^{m}$ as $\partial_sf = |\partial_{\rho}\vec X^m|^{-1}\partial_{\rho}f$.

The above scheme is ``weakly implicit'' with only one nonlinear term introduced in \eqref{eqn:aniso SP-PFEM eq1} and
\eqref{eqn:aniso SP-PFEM eq2}, respectively. In particular, the nonlinear term is a polynomial function of degree at most two with respect to the components of $\boldsymbol{X}^{m+1}$ and $\mu^{m+1}$.
Again, similar to \cite{bao2021structurepreserving} for surface diffusion,
the fully-implicit SP-PFEM \eqref{eqn:aniso SP-PFEM} can be efficiently and accurately solved by the Newton's iterative method in practical computations.

\begin{remark}The choice of $\boldsymbol{n}^{m+\frac{1}{2}}$ in \eqref{eqn:aniso SP-PFEM} plays an essential role in the proof of the area conservation, but it makes the numerical scheme fully-implicit, i.e. a nonlinear system has to be solved at each time step. By replacing $\boldsymbol{n}^{m+1/2}$ with $\boldsymbol{n}^m$, we can easily construct a semi-implicit PFEM, where only a linear system has to be solved at each time step. Similar to the fully-implicit SP-PFEM \eqref{eqn:aniso SP-PFEM}, the semi-implicit PFEM can also be proved to be unconditionally energy-stable if $\gamma(\boldsymbol{n})$ satisfies the condition \eqref{engstabgmp}. Of course, the semi-implicit PFEM does not conserve the area at the fully-discrete level.
\end{remark}

\subsection{Main results}
Let $A^m$ be the area of the interior region of the piecewise linear closed curve $\Gamma^{m}$, and $W^m_c$ ($m\ge0$) be its
energy, which are defined as
\begin{equation}\label{total area, SP-PFEM}
A^m:=\frac{1}{2}\sum_{j=1}^N\left(x_j^m-x_{j-1}^{m}\right)
\left(y_j^m+y_{j-1}^{m}\right),
\quad W_c^m:=W_c(\Gamma^m)=\sum_{j=1}^N|\boldsymbol{h}_j^m
|\gamma(\boldsymbol{n}_j^m).
\end{equation}

Denote
\begin{equation}\label{Fnnhat}
F(\boldsymbol{n},\hat{\boldsymbol{n}})=
\frac{\gamma(\hat{\boldsymbol{n}})^2
-\gamma(\boldsymbol{n})^2+2\gamma(\boldsymbol{n})(\boldsymbol{\xi}
\cdot\hat{\boldsymbol{n}}^{\perp})(\boldsymbol{n}\cdot
\hat{\boldsymbol{n}}^{\perp})}{\gamma(\boldsymbol{n})(\boldsymbol{n}
\cdot\hat{\boldsymbol{n}}^\perp)^2}, \qquad \forall \boldsymbol{n}\ne \pm\hat{\boldsymbol{n}}\in{\mathbb S}^1,
\end{equation}
and define the minimal stabilizing function $k_0(\boldsymbol{n}):\ {\mathbb S}^1\to {\mathbb R}^+$ as (the existence will be given in next section)
\begin{equation}\label{existence of k_0(n)}
k_0(\boldsymbol{n}):=\max_{\hat{\boldsymbol{n}}\in \mathbb{S}^1_{\boldsymbol{n}}} F(\boldsymbol{n},\hat{\boldsymbol{n}}),
\qquad \hbox{with} \quad \mathbb{S}^1_{\boldsymbol{n}}:=\left\{\hat{\boldsymbol{n}}\in \mathbb{S}^1\ |\ \hat{\boldsymbol{n}}\cdot\boldsymbol{n}\geq 0\right\}, \qquad \boldsymbol{n}\in {\mathbb S}^1.
\end{equation}
Then for the SP-PFEM \eqref{eqn:aniso SP-PFEM}, we have
\begin{theorem}[structure-preserving]
Assume $\gamma(\boldsymbol{n})$ satisfies \eqref{engstabgmp} and take
$k(\boldsymbol{n})$ in \eqref{def of the energy matrix Z_k(n)} satisfying $k(\boldsymbol{n})\geq k_0(\boldsymbol{n})$ for $\boldsymbol{n}\in {\mathbb S}^1$, then the SP-PFEM \eqref{eqn:aniso SP-PFEM} is area conservation and
energy dissipation, i.e.
\begin{equation}\label{area conservation, SP-PFEM}
  A^m\equiv A^0=\frac{1}{2}\sum_{j=1}^N\left(x_j^0-x_{j-1}^0\right)
  \left(y_j^0+y_{j-1}^0\right),\qquad m\geq 0.
\end{equation}
\begin{equation}\label{engdpfd2}
W^{m+1}_c\leq W^m_c\leq  \ldots \le W^0_c=\sum_{j=1}^N|\boldsymbol{h}_j^0|\,\gamma(\boldsymbol{n}_j^0),\qquad \forall m\ge0.
\end{equation}
\end{theorem}

The proof of area conservation \eqref{area conservation, SP-PFEM} is similar to the proof in \cite[Theorem 2.1]{bao2021structurepreserving} and it is omitted here for brevity, and we will establish the energy dissipation or
unconditional energy stability  \eqref{engdpfd2} in next section.

\section{Energy dissipation}
In this section, we first show,  under the condition \eqref{engstabgmp}
on $\gamma(\boldsymbol{n})$, the minimal stabilizing function $k_0(\boldsymbol{n})$ \eqref{existence of k_0(n)} is well defined, and then prove the energy dissipation of the SP-PFEM
\eqref{eqn:aniso SP-PFEM}.

\subsection{Choice of the stabilizing function} The function $F(\boldsymbol{n}, \hat{\boldsymbol{n}})$ is continuous for $\boldsymbol{n}\neq \pm \hat{\boldsymbol{n}}$. Thus to show the maximum in \eqref{existence of k_0(n)} is finite, it suffices to extent the definition of $F(\boldsymbol{n}, \hat{\boldsymbol{n}})$ to $\boldsymbol{n}=\pm \hat{\boldsymbol{n}}$.


\begin{theorem}[existence of limit]
For $\gamma(\boldsymbol{p})\in C^2(\mathbb{R}^2\setminus \{{\bf 0}\})$, we have
\begin{equation}\label{limFnnh}
\lim_{\substack{\hat{\boldsymbol{n}}\to\boldsymbol{n}\\ \hat{\boldsymbol{n}}\in {\mathbb S}^1}} F(\boldsymbol{n},\hat{\boldsymbol{n}})
=(\boldsymbol{n}^\perp)^T \bf{H}_\gamma(\boldsymbol{n})\boldsymbol{n}^\perp+
\frac{|\boldsymbol{\xi}|^2}
{\gamma(\boldsymbol{n})},\qquad  \forall \boldsymbol{n}\in{\mathbb S}^1.
\end{equation}
\end{theorem}
\begin{proof}
Plugging the vector decomposition
$\gamma(\boldsymbol{n})=\boldsymbol{\xi}\cdot \boldsymbol{n}=(\boldsymbol{\xi} \cdot \hat{\boldsymbol{n}}^\perp)(\boldsymbol{n}\cdot \hat{\boldsymbol{n}}^\perp)+(\boldsymbol{\xi}\cdot \hat{\boldsymbol{n}})(\boldsymbol{n}\cdot\hat{\boldsymbol{n}})$
and $1=\boldsymbol{n}\cdot \boldsymbol{n}=(\boldsymbol{n} \cdot \hat{\boldsymbol{n}}^\perp)^2+(\boldsymbol{n}\cdot \hat{\boldsymbol{n}})^2$ into \eqref{Fnnhat}, we get
\begin{align}\label{Fnnhlim67}
F(\boldsymbol{n},\hat{\boldsymbol{n}})&=\frac{\gamma(\hat{\boldsymbol{n}})^2
-\gamma(\boldsymbol{n})^2+2\gamma(\boldsymbol{n})^2-
2\gamma(\boldsymbol{n})(\boldsymbol{\xi}\cdot\hat{\boldsymbol{n}})
(\boldsymbol{n}\cdot\hat{\boldsymbol{n}})}{\gamma(\boldsymbol{n})|
\boldsymbol{n}-\hat{\boldsymbol{n}}|^2(1-|\boldsymbol{n}-
\hat{\boldsymbol{n}}|^2/4)}\nonumber\\
&=\frac{\gamma(\hat{\boldsymbol{n}})^2+\gamma(\boldsymbol{n})^2-
2\gamma(\boldsymbol{n})(\boldsymbol{\xi}\cdot\hat{\boldsymbol{n}})
(1-|\boldsymbol{n}-\hat{\boldsymbol{n}}|^2/2)}{\gamma(\boldsymbol{n})
|\boldsymbol{n}-\hat{\boldsymbol{n}}|^2(1-|\boldsymbol{n}-
\hat{\boldsymbol{n}}|^2/4)}\nonumber\\
&=\frac{1}{1-|\boldsymbol{n}-\hat{\boldsymbol{n}}|^2/4}
\left[\frac{\gamma(\hat{\boldsymbol{n}})^2-\gamma(\boldsymbol{n})^2
-2\gamma(\boldsymbol{n})(\boldsymbol{\xi}\cdot(\hat{\boldsymbol{n}}
-\boldsymbol{n}))}
{\gamma(\boldsymbol{n})|\boldsymbol{n}-\hat{\boldsymbol{n}}|^2}
+\boldsymbol{\xi}\cdot
\hat{\boldsymbol{n}}\right].
\end{align}
Here we use the following equality
\[\boldsymbol{n}\cdot \hat{\boldsymbol{n}}=
\frac{|\boldsymbol{n}|^2+
|\hat{\boldsymbol{n}}|^2-|\boldsymbol{n}-
\hat{\boldsymbol{n}}|^2}{2}=1-\frac{|\boldsymbol{n}-
\hat{\boldsymbol{n}}|^2}{2}.\]\color{black}
Under the condition $\gamma(\boldsymbol{p})\in C^2(\mathbb{R}^2\setminus \{{\bf 0}\})$, using Taylor expansion and noting $\nabla\gamma(\boldsymbol{p})^2=2\gamma(\boldsymbol{p})
\nabla\gamma(\boldsymbol{p})$ and $\boldsymbol{\xi}=\nabla \gamma(\boldsymbol{p})|_{\boldsymbol{p}=\boldsymbol{n}}$,  we obtain
\begin{equation*}
\gamma(\boldsymbol{p})^2-\gamma(\boldsymbol{n})^2-
2\gamma(\boldsymbol{n})\boldsymbol{\xi}\cdot
(\boldsymbol{p}-\boldsymbol{n})=(\boldsymbol{p}-\boldsymbol{n})^T
\left[\gamma(\boldsymbol{n}){\bf H}_\gamma(\boldsymbol{n})+
\boldsymbol{\xi}\boldsymbol{\xi}^T\right](\boldsymbol{p}-\boldsymbol{n})
+o(|\boldsymbol{p}-\boldsymbol{n}|^2).
\end{equation*}
For any $\boldsymbol{n}\in{\mathbb S}^1$, noting that
\[ \lim\limits_{\substack{\boldsymbol{p}\to\boldsymbol{n}^+\\ \boldsymbol{p}\in \mathbb{S}^1}}\frac{\boldsymbol{p}-\boldsymbol{n}}
{|\boldsymbol{p}-\boldsymbol{n}|}=\boldsymbol{n}^\perp,\quad
\lim\limits_{\substack{\boldsymbol{p}\to\boldsymbol{n}^-\\ \boldsymbol{p}\in \mathbb{S}^1}}\frac{\boldsymbol{p}-\boldsymbol{n}}
{|\boldsymbol{p}-\boldsymbol{n}|}=-\boldsymbol{n}^\perp,
\]
where $\boldsymbol{p}\to \boldsymbol{n}^+/\boldsymbol{n}^-$ means $\boldsymbol{p}\cdot \boldsymbol{n}^\perp\geq 0 /\leq 0$, respectively. We then get
\begin{align}
 \lim\limits_{\substack{\boldsymbol{p}\to\boldsymbol{n}^+\\ \boldsymbol{p}\in \mathbb{S}^1}}\frac{(\boldsymbol{p}-\boldsymbol{n})^T
\left[\gamma(\boldsymbol{n}){\bf H}_\gamma(\boldsymbol{n})+
\boldsymbol{\xi}\boldsymbol{\xi}^T\right](\boldsymbol{p}-\boldsymbol{n})}
{|\boldsymbol{p}-\boldsymbol{n}|^2}&=(\boldsymbol{n}^\perp)^T\left[\gamma(\boldsymbol{n}){\bf H}_\gamma(\boldsymbol{n})+
\boldsymbol{\xi}\boldsymbol{\xi}^T\right]\boldsymbol{n}^\perp,\nonumber\\
\lim\limits_{\substack{\boldsymbol{p}\to\boldsymbol{n}^-\\ \boldsymbol{p}\in \mathbb{S}^1}}\frac{(\boldsymbol{p}-\boldsymbol{n})^T
\left[\gamma(\boldsymbol{n}){\bf H}_\gamma(\boldsymbol{n})+
\boldsymbol{\xi}\boldsymbol{\xi}^T\right](\boldsymbol{p}-\boldsymbol{n})}
{|\boldsymbol{p}-\boldsymbol{n}|^2}&=-(\boldsymbol{n}^\perp)^T\left[\gamma(\boldsymbol{n}){\bf H}_\gamma(\boldsymbol{n})+
\boldsymbol{\xi}\boldsymbol{\xi}^T\right](-\boldsymbol{n}^\perp)\nonumber\\
&=(\boldsymbol{n}^\perp)^T\left[\gamma(\boldsymbol{n}){\bf H}_\gamma(\boldsymbol{n})+
\boldsymbol{\xi}\boldsymbol{\xi}^T\right]\boldsymbol{n}^\perp,\nonumber
\end{align}
thus we have
\begin{align}\label{Fplimit}
\lim_{\substack{\boldsymbol{p}\to\boldsymbol{n}\\ \boldsymbol{p}\in \mathbb{S}^1}}\frac{\gamma(\boldsymbol{p})^2-
\gamma(\boldsymbol{n})^2-2\gamma(\boldsymbol{n})
\boldsymbol{\xi}\cdot(\boldsymbol{p}-\boldsymbol{n})}
{|\boldsymbol{p}-\boldsymbol{n}|^2}&=(\boldsymbol{n}^\perp)^T\left[\gamma(\boldsymbol{n}){\bf H}_\gamma(\boldsymbol{n})+
\boldsymbol{\xi}\boldsymbol{\xi}^T\right]\boldsymbol{n}^\perp\nonumber\\
&=\gamma(\boldsymbol{n})\,
(\boldsymbol{n}^\perp)^T \bf{H}_\gamma(\boldsymbol{n})\boldsymbol{n}^\perp+
(\boldsymbol{\xi}\cdot\boldsymbol{n}^\perp)^2.
\end{align}
Combining \eqref{Fnnhlim67} and \eqref{Fplimit}, noting
\eqref{The xi-vector} to get $\gamma(\boldsymbol{n})=
\boldsymbol{\xi}\cdot
\boldsymbol{n}$, we obtain
\begin{eqnarray}\label{limFnnhh}
\lim\limits_{\substack{\hat{\boldsymbol{n}}\to\boldsymbol{n}\\ \hat{\boldsymbol{n}}\in {\mathbb S}^1}}
F(\boldsymbol{n},\hat{\boldsymbol{n}})&=&\frac{1}{\gamma(\boldsymbol{n})}
\lim_{\substack{\boldsymbol{p}\to\boldsymbol{n}\\\boldsymbol{p}\in \mathbb{S}^1}}\frac{\gamma(\boldsymbol{p})^2-
\gamma(\boldsymbol{n})^2-2\gamma(\boldsymbol{n})
\boldsymbol{\xi}\cdot(\boldsymbol{p}-\boldsymbol{n})}
{|\boldsymbol{p}-\boldsymbol{n}|^2}
+\boldsymbol{\xi}\cdot
\boldsymbol{n}\nonumber\\
&=&(\boldsymbol{n}^\perp)^T \bf{H}_\gamma(\boldsymbol{n})\boldsymbol{n}^\perp+
\frac{(\boldsymbol{\xi}\cdot\boldsymbol{n}^\perp)^2}{\gamma(\boldsymbol{n})}
+\boldsymbol{\xi}\cdot
\boldsymbol{n}\nonumber\\
&=&(\boldsymbol{n}^\perp)^T \bf{H}_\gamma(\boldsymbol{n})\boldsymbol{n}^\perp+
\frac{|\boldsymbol{\xi}|^2}{\gamma(\boldsymbol{n})}.
\end{eqnarray}
The proof is completed.
\end{proof}

Under the condition \eqref{engstabgmp}, for any $\boldsymbol{n}\in {\mathbb S}^1$, it is easy to see that $F(\boldsymbol{n},\hat{\boldsymbol{n}})$
is a continuous function for $\hat{\boldsymbol{n}}\in {\mathbb S}^1$
with $\hat{\boldsymbol{n}}\ne -\boldsymbol{n}$. Furthermore, if $\gamma(\boldsymbol{n})=\gamma(-\boldsymbol{n})$, then we know $F(\boldsymbol{n},\hat{\boldsymbol{n}})\in C^1(\mathbb{S}^1\times \mathbb{S}^1)$. This, together with
the above Theorem, suggests us to define the following

\begin{theorem}[existence of stabilizing function]
\label{energy dissipation condition and k_0}
Under the condition \eqref{engstabgmp} on $\gamma(\boldsymbol{n})$ and assume $k(\boldsymbol{n})\geq k_0(\boldsymbol{n})$ for $\boldsymbol{n}\in{\mathbb S}^1$ in \eqref{def of the energy matrix Z_k(n)},
we have
\begin{equation}\label{energy dissipation condition on Z_k, ineq}
  \gamma(\boldsymbol{n})[(\hat{\boldsymbol{n}}^{\perp})^T \boldsymbol{Z}_k(\boldsymbol{n})\hat{\boldsymbol{n}}^{\perp}]\geq \gamma(\hat{\boldsymbol{n}})^2,\qquad \forall \boldsymbol{n},\, \hat{\boldsymbol{n}}\in \mathbb{S}^1.
\end{equation}
In addition, we have an alternative definition of $k_0(\boldsymbol{n})$ in
\eqref{existence of k_0(n)} as
\begin{equation}\label{alternative definition of k_0}
k_0(\boldsymbol{n})=\inf\left\{k(\boldsymbol{n})\ |\  \gamma(\boldsymbol{n})[(\hat{\boldsymbol{n}}^{\perp})^T \boldsymbol{Z}_k(\boldsymbol{n})\hat{\boldsymbol{n}}^{\perp}]\geq \gamma(\hat{\boldsymbol{n}})^2,\quad \forall  \hat{\boldsymbol{n}}\in \mathbb{S}^1\right\}, \quad \boldsymbol{n}\in \mathbb{S}^1.
\end{equation}
\end{theorem}
\begin{proof}
Assume $k(\boldsymbol{n})\geq k_0(\boldsymbol{n})$ for $\boldsymbol{n}\in{\mathbb S}^1$. For any $\boldsymbol{n}\in \mathbb{S}^1$, when $\hat{\boldsymbol{n}}\in \mathbb{S}^1_{\boldsymbol{n}}$, i.e. $\hat{\boldsymbol{n}}\cdot\boldsymbol{n}\geq 0$, plugging
 \eqref{def of the energy matrix Z_k(n)} into the left hand of
 \eqref{energy dissipation condition on Z_k, ineq}, noting \eqref{Fnnhat}
 and \eqref{existence of k_0(n)},  we have
\begin{align}\label{knk0n567}
\gamma(\boldsymbol{n})[(\hat{\boldsymbol{n}}^{\perp})^T \boldsymbol{Z}_k(\boldsymbol{n})\hat{\boldsymbol{n}}^{\perp}]
&=\gamma(\boldsymbol{n})^2-2\gamma(\boldsymbol{n})(\boldsymbol{\xi}\cdot \hat{\boldsymbol{n}}^\perp)(\boldsymbol{n}\cdot\hat{\boldsymbol{n}}^\perp)
+\gamma(\boldsymbol{n})k(\boldsymbol{n})(\boldsymbol{n}\cdot \hat{\boldsymbol{n}}^\perp)^2\nonumber\\
&\ge\gamma(\boldsymbol{n})^2-2\gamma(\boldsymbol{n})(\boldsymbol{\xi}\cdot \hat{\boldsymbol{n}}^\perp)(\boldsymbol{n}\cdot\hat{\boldsymbol{n}}^\perp)
+\gamma(\boldsymbol{n})k_0(\boldsymbol{n})(\boldsymbol{n}\cdot \hat{\boldsymbol{n}}^\perp)^2\nonumber\\
&\ge\gamma(\boldsymbol{n})^2-2\gamma(\boldsymbol{n})(\boldsymbol{\xi}\cdot \hat{\boldsymbol{n}}^\perp)(\boldsymbol{n}\cdot\hat{\boldsymbol{n}}^\perp)
+\gamma(\boldsymbol{n})F(\boldsymbol{n},\hat{\boldsymbol{n}})(\boldsymbol{n}\cdot \hat{\boldsymbol{n}}^\perp)^2\nonumber\\
&=\gamma(\hat{\boldsymbol{n}})^2.
\end{align}
On the other hand, when $\hat{\boldsymbol{n}}\cdot\boldsymbol{n}< 0$,
then $-\hat{\boldsymbol{n}}\cdot\boldsymbol{n}> 0$, from
\eqref{knk0n567} by replacing $\hat{\boldsymbol{n}}$ by $-\hat{\boldsymbol{n}}$ and noting $\gamma(-\hat{\boldsymbol{n}})=\gamma(\hat{\boldsymbol{n}})$, we have
\begin{equation}\label{knk0n589}
\gamma(\boldsymbol{n})[(\hat{\boldsymbol{n}}^{\perp})^T \boldsymbol{Z}_k(\boldsymbol{n})\hat{\boldsymbol{n}}^{\perp}]=
\gamma(\boldsymbol{n})[(-\hat{\boldsymbol{n}}^{\perp})^T \boldsymbol{Z}_k(\boldsymbol{n})(-\hat{\boldsymbol{n}}^{\perp})]\ge
\gamma(-\hat{\boldsymbol{n}})^2=\gamma(\hat{\boldsymbol{n}})^2.
\end{equation}
Combining \eqref{knk0n567} and \eqref{knk0n589}, we get
\eqref{energy dissipation condition on Z_k, ineq} immediately.

From the above proof, it is easy to see that
\[\gamma(\boldsymbol{n})[(\hat{\boldsymbol{n}}^{\perp})^T \boldsymbol{Z}_{k_0}(\boldsymbol{n})\hat{\boldsymbol{n}}^{\perp}]\geq \gamma(\hat{\boldsymbol{n}})^2,\qquad \forall \boldsymbol{n},\, \hat{\boldsymbol{n}}\in \mathbb{S}^1,\]
which implies
\begin{equation}\label{k0gret}
k_0(\boldsymbol{n})\geq \inf\left\{k(\boldsymbol{n})\ |\  \gamma(\boldsymbol{n})[(\hat{\boldsymbol{n}}^{\perp})^T \boldsymbol{Z}_k(\boldsymbol{n})\hat{\boldsymbol{n}}^{\perp}]\geq \gamma(\hat{\boldsymbol{n}})^2,\quad \forall \hat{\boldsymbol{n}}\in \mathbb{S}^1\right\},\quad \forall \boldsymbol{n}\in \mathbb{S}^1.
\end{equation}
On the other hand, suppose $\boldsymbol{Z}_k(\boldsymbol{n})$ satisfies \eqref{energy dissipation condition on Z_k, ineq}, then we have
\begin{equation}\label{eq: temp1 in energy stable on Z_k}
  \gamma(\boldsymbol{n})\left(\gamma(\boldsymbol{n})-2(\boldsymbol{\xi}\cdot \hat{\boldsymbol{n}}^\perp)(\boldsymbol{n}\cdot\hat{
  \boldsymbol{n}}^\perp)+k(\boldsymbol{n})(\boldsymbol{n}\cdot \hat{\boldsymbol{n}}^\perp)^2\right)\geq \gamma(\hat{\boldsymbol{n}})^2,\qquad \forall \hat{\boldsymbol{n}}\in \mathbb{S}^1_{\boldsymbol{n}},
\end{equation}
which implies
\begin{equation}
k(\boldsymbol{n})\geq \frac{\gamma(\hat{\boldsymbol{n}})^2-\gamma(\boldsymbol{n})^2
+2\gamma(\boldsymbol{n})(\boldsymbol{\xi}\cdot\hat{
\boldsymbol{n}}^{\perp})(\boldsymbol{n}\cdot\hat{
\boldsymbol{n}}^{\perp})}{\gamma(\boldsymbol{n})
(\boldsymbol{n}\cdot\hat{\boldsymbol{n}}^\perp)^2}=
F(\boldsymbol{n},\hat{\boldsymbol{n}}),\quad \forall \hat{\boldsymbol{n}}\in \mathbb{S}^1_{\boldsymbol{n}}.
\end{equation}
By condition \eqref{engstabgmp}, this inequality holds for all $\hat{\boldsymbol{n}}\in \mathbb{S}^1$. Thus we get $k(\boldsymbol{n})\geq k_0(\boldsymbol{n})$, which implies
\begin{equation}\label{k0less}
k_0(\boldsymbol{n})\leq  \inf\left\{k(\boldsymbol{n})\ |\  \gamma(\boldsymbol{n})[(\hat{\boldsymbol{n}}^{\perp})^T \boldsymbol{Z}_k(\boldsymbol{n})\hat{\boldsymbol{n}}^{\perp}]\geq \gamma(\hat{\boldsymbol{n}})^2,\quad \forall \hat{\boldsymbol{n}}\in \mathbb{S}^1\right\},\quad \forall \boldsymbol{n}\in \mathbb{S}^1.
\end{equation}
Combining \eqref{k0gret} and \eqref{k0less}, we obtain
\eqref{alternative definition of k_0} immediately.
\end{proof}

\begin{remark}\label{rmk 4.1}
Assume $\boldsymbol{n}=(-\sin\theta,\cos\theta)^T$ ($\theta\in[-\pi,\pi]$) and $\hat{\boldsymbol{n}}=(-\sin\hat{\theta},\cos\hat{\theta})^T$, then the problem to find the minimal stabilizing function $k_0(\boldsymbol{n})$ defined in \eqref{existence of k_0(n)} can be reformulated as an optimization problem in term of the single variable $\hat{\theta}$, i.e.,
\begin{equation}\label{opti}
\tilde{k}_0(\theta):=	k_0(\boldsymbol{n})=k_0(-\sin\theta,\cos\theta)=
\max_{\hat{\theta}\in [\theta-\frac{\pi}{2},\theta+\frac{\pi}{2}] }
\tilde{F}^\theta(\hat{\theta}), \qquad -\pi\le \theta\le \pi,
\end{equation}
where
\begin{equation}\label{tildeF}
\tilde{F}^\theta(\hat{\theta}):=F(\boldsymbol{n},\hat{\boldsymbol{n}})
=\frac{\hat{\gamma}(\hat{\theta})^2-\hat{\gamma}(\theta)^2-2\hat{\gamma}(\theta)\hat{\gamma}'(\theta)\cos(\hat{\theta}-\theta)\sin(\hat{\theta}-\theta)}{\hat{\gamma}(\theta)\sin^2(\hat{\theta}-\theta)}+ 2\hat{\gamma}(\theta),
\end{equation}
with $\hat{\gamma}(\theta):=\gamma(\boldsymbol{n})
=\gamma(-\sin\theta,\cos\theta)$ and
$\hat{\gamma}(\hat{\theta}):=\gamma(\hat{\boldsymbol{n}})=
\gamma(-\sin\hat{\theta},\cos\hat{\theta})$ by noting
$\boldsymbol{\xi}=\boldsymbol{\xi}(\boldsymbol{n})
=\hat{\gamma}(\theta)\boldsymbol{n}-\hat{\gamma}'(\theta)\boldsymbol{n}^\perp$.
Thus for a given $\boldsymbol{n}$ (or $\theta$), we can obtain $k_0(\boldsymbol{n})$ (or $\tilde{k}_0(\theta)$) by numerically solving the above single-variable optimization problem \eqref{opti}.
\end{remark}

\begin{corollary}[positivity of the minimal stabilizing function] Assume \eqref{energy dissipation condition on Z_k, ineq}
is satisfied, then $\boldsymbol{Z}_k(\boldsymbol{n})$ is a symmetric positive definite matrix and
\begin{equation}\label{gmnmn45}
\gamma(-\boldsymbol{n})=\gamma(\boldsymbol{n}), \qquad
k_0(\boldsymbol{n})>0, \qquad \forall \boldsymbol{n}\in \mathbb{S}^1.
\end{equation}
\end{corollary}
\begin{proof}
Taking $\hat{\boldsymbol{n}}=-\boldsymbol{n}$ in
\eqref{energy dissipation condition on Z_k, ineq}, noting
the first equality in \eqref{knk0n567}, we get $\gamma(\boldsymbol{n})^2\ge \gamma(-\boldsymbol{n})^2$ which suggests $\gamma(-\boldsymbol{n})^2\ge \gamma(-(-\boldsymbol{n}))^2=\gamma(\boldsymbol{n})^2$, and thus we
obtain the first equality in \eqref{gmnmn45} since $\gamma(\boldsymbol{n})>0$. From
\eqref{energy dissipation condition on Z_k, ineq}, we get $\boldsymbol{Z}_{k}(\boldsymbol{n})$ is symmetric positive definite,
which implies $k(\boldsymbol{n})={\rm Tr}(\boldsymbol{Z}_{k}(\boldsymbol{n}))\ge
k_0(\boldsymbol{n})={\rm Tr}(\boldsymbol{Z}_{k_0}(\boldsymbol{n}))>0$ for
$\boldsymbol{n}\in \mathbb{S}^1$.
\end{proof}

If we consider from the anisotropic surface energy $\gamma(\boldsymbol{n})$
to its corresponding minimal stabilizing function $k_0(\boldsymbol{n})$ defined in \eqref{alternative definition of k_0} (or
\eqref{existence of k_0(n)}) as a mapping, then it is a sub-linear mapping,
i.e., positively homogeneous and subadditive.

\begin{lemma}[positive homogeneity and subadditivity]\label{lemma: sub-linear}
Assume $k_0(\boldsymbol{n})$, $k_1(\boldsymbol{n})$ and $k_2(\boldsymbol{n})$
be the minimal stabilizing functions for the anisotropic surface energies $\gamma(\boldsymbol{n})$, $\gamma_1(\boldsymbol{n})$ and $\gamma_2(\boldsymbol{n})$, respectively, then we have

(i) if $\gamma_1(\boldsymbol{n})=c\,\gamma(\boldsymbol{n})$ with $c>0$,
then $k_1(\boldsymbol{n})=c\,k_0(\boldsymbol{n})$ for $\boldsymbol{n}\in \mathbb{S}^1$, and

(ii) if $\gamma(\boldsymbol{n})=\gamma_1(\boldsymbol{n})+
\gamma_2(\boldsymbol{n})$, then $k_0(\boldsymbol{n})\le k_1(\boldsymbol{n})+
k_2(\boldsymbol{n})$ for $\boldsymbol{n}\in \mathbb{S}^1$.
\end{lemma}
\begin{proof}
From \eqref{The xi-vector}, we get
\begin{equation}\label{tildexi1}
\boldsymbol{\xi}=\nabla \gamma(\boldsymbol{p})\big|_{\boldsymbol{p}=\boldsymbol{n}},\qquad
\boldsymbol{\xi}_1=\nabla \gamma_1(\boldsymbol{p})\big|_{\boldsymbol{p}=\boldsymbol{n}},\qquad
\boldsymbol{\xi}_2=\nabla \gamma_2(\boldsymbol{p})\big|_{\boldsymbol{p}=\boldsymbol{n}}.
\end{equation}

(i) If $\gamma_1(\boldsymbol{n})=c\,\gamma(\boldsymbol{n})$, we get
$\boldsymbol{\xi}_1=c\,\boldsymbol{\xi}$. This, together with \eqref{Fnnhat}, implies
\begin{equation}\label{tildeFnn}
F_1(\boldsymbol{n},\hat{\boldsymbol{n}})=
\frac{\gamma_1(\hat{\boldsymbol{n}})^2
-\gamma_1(\boldsymbol{n})^2+2\gamma_1(\boldsymbol{n})(
\boldsymbol{\xi}_1
\cdot\hat{\boldsymbol{n}}^{\perp})(\boldsymbol{n}\cdot
\hat{\boldsymbol{n}}^{\perp})}{\gamma_1(\boldsymbol{n})(\boldsymbol{n}
\cdot\hat{\boldsymbol{n}}^\perp)^2}
=c\,F(\boldsymbol{n},\hat{\boldsymbol{n}}).
\end{equation}
Combining \eqref{tildeFnn} and \eqref{existence of k_0(n)}, we obtain
the positive homogeneity immediately.

(ii) If $\gamma(\boldsymbol{n})=\gamma_1(\boldsymbol{n})+
\gamma_2(\boldsymbol{n})$, then $\boldsymbol{\xi}=\boldsymbol{\xi}_1+\boldsymbol{\xi}_2$, thus
we have
\begin{align*}
\boldsymbol{Z}_{k_{1}+k_{2}}(\boldsymbol{n})&=\gamma(\boldsymbol{n})I_2
-\boldsymbol{\xi}
 \boldsymbol{n}^T-\boldsymbol{n}\boldsymbol{\xi}^T+
(k_{1}(\boldsymbol{n})+k_{2}(\boldsymbol{n}))
\boldsymbol{n}\boldsymbol{n}^T\\
&=\boldsymbol{Z}_{k_{1}}^{(1)}(\boldsymbol{n})+
\boldsymbol{Z}_{k_{2}}^{(2)}(\boldsymbol{n}),
\end{align*}
where
\begin{align*}
&\boldsymbol{Z}_{k_{1}}^{(1)}(\boldsymbol{n})=\gamma_1(\boldsymbol{n})I_2-
\boldsymbol{\xi}_1\boldsymbol{n}^T-\boldsymbol{n}(\boldsymbol{\xi}_1)^T
+k_{1}(\boldsymbol{n})
  \boldsymbol{n}\boldsymbol{n}^T,\\
&\boldsymbol{Z}_{k_{2}}^{(2)}(\boldsymbol{n})=\gamma_2(\boldsymbol{n})I_2-
\boldsymbol{\xi}_2\boldsymbol{n}^T-\boldsymbol{n}(\boldsymbol{\xi}_2)^T
+k_{2}(\boldsymbol{n})
  \boldsymbol{n}\boldsymbol{n}^T.
\end{align*}
By using Cauchy inequality, we get
\begin{align}\label{gam1221b}
&\gamma(\boldsymbol{n})[(\hat{\boldsymbol{n}}^{\perp})^T \boldsymbol{Z}_{k_{1}+k_{2}}(\boldsymbol{n})\hat{\boldsymbol{n}}^{\perp}]
\nonumber\\
&\geq\left(\sqrt{\gamma_1(\boldsymbol{n})[(\hat{\boldsymbol{n}}^{\perp}
)^T\boldsymbol{Z}_{k_{1}}^{(1)}(\boldsymbol{n})\hat{\boldsymbol{n}}^{\perp}]
}+\sqrt{\gamma_2(\boldsymbol{n})[(\hat{\boldsymbol{n}}^{\perp})^T
\boldsymbol{Z}_{k_{2}}^{(2)}
(\boldsymbol{n})\hat{\boldsymbol{n}}^{\perp}]}\right)^2\nonumber\\
&\geq\left(\gamma_1(\hat{\boldsymbol{n}})+
\gamma_2(\hat{\boldsymbol{n}})\right)^2=\gamma(\hat{\boldsymbol{n}})^2.
\end{align}
Combining \eqref{gam1221b} and \eqref{alternative definition of k_0}, we get $k_0(\boldsymbol{n})\leq k_{1}(\boldsymbol{n})+k_{2}(\boldsymbol{n})$ for $\boldsymbol{n}\in \mathbb{S}^1$.
\end{proof}

\subsection{Energy dissipation}
For the SP-PFEM \eqref{eqn:aniso SP-PFEM} , we have:
\begin{theorem}[energy dissipation]
\label{energy dissipation condition, theorem}
Assume the surface energy matrix $\boldsymbol{Z}_k(\boldsymbol{n})$ satisfies
\eqref{energy dissipation condition on Z_k, ineq},
then the SP-PFEM \eqref{eqn:aniso SP-PFEM} is unconditionally energy stable, i.e. for any $\tau>0$, we have
\begin{equation}\label{engdpfd}
W^{m+1}_c\leq W^m_c\leq  \ldots \le W^0_c=\sum_{j=1}^N|\boldsymbol{h}_j^0|\,\gamma(\boldsymbol{n}_j^0),\qquad \forall m\ge0.
\end{equation}
\end{theorem}
\begin{proof}
Under \eqref{energy dissipation condition on Z_k, ineq}, we know that
$\boldsymbol{Z}_k(\boldsymbol{n})$ is symmetric positive definite. Thus we have
\begin{equation}\label{Zkuvm}	
\Bigl(\boldsymbol{Z}_k(\boldsymbol{n})\boldsymbol{u},\boldsymbol{u}
-\boldsymbol{v}\Bigr)\geq \frac{1}{2}\Bigl(\boldsymbol{Z}_k(\boldsymbol{n})\boldsymbol{u},
\boldsymbol{u}\Bigr)-\frac{1}{2}\Bigl(\boldsymbol{Z}_k(\boldsymbol{n})
\boldsymbol{v},\boldsymbol{v}\Bigr), \quad\forall \boldsymbol{u},\boldsymbol{v}\in{\mathbb R}^2.
\end{equation}
 Using
\eqref{reformulation of Zk, auxillary eq 2} and $\boldsymbol{\xi}\cdot\boldsymbol{n}=\gamma(\boldsymbol{n})$, we get
\begin{equation}\label{auxillary eq, energy dissipation} (\partial_s\boldsymbol{X}^m)^T\boldsymbol{Z}_k(\boldsymbol{n}^m)
\partial_s\boldsymbol{X}^m=\boldsymbol{\tau}^m\cdot (\boldsymbol{\xi}^m)^{\perp}=\gamma(\boldsymbol{n}^m).
\end{equation}
Combining \eqref{auxillary eq, energy dissipation} and \eqref{Zkuvm},
noting $\boldsymbol{Z}_k(\boldsymbol{n})$ satisfies \eqref{energy dissipation condition on Z_k, ineq}, we obtain	
\begin{align}\label{difference between two time steps}
&\Bigl(\boldsymbol{Z}_k(\boldsymbol{n}^m)\partial_s \boldsymbol{X}^{m+1},~\partial_s\boldsymbol{X}^{m+1}-
\partial_s\boldsymbol{X}^m\Bigr)_{\Gamma^m}^h+
\int_{\Gamma^m}\gamma(\boldsymbol{n}^m)ds\nonumber\\
&\geq\frac{1}{2}\Bigl(\boldsymbol{Z}_k(\boldsymbol{n}^m)\partial_s \boldsymbol{X}^{m+1},~\partial_s\boldsymbol{X}^{m+1}
\Bigr)_{\Gamma^m}^h+\frac{1}{2}\int_{\Gamma^m}\gamma(
\boldsymbol{n}^m)ds\nonumber\\
&=\sum_{j=1}^N\frac{\left(\boldsymbol{h}_j^{m+1}\right)^T
\boldsymbol{Z}_k(\boldsymbol{n}_j^m)\boldsymbol{h}_j^{m+1}+\gamma(
\boldsymbol{n}_j^m)|\boldsymbol{h}_j^m|^2}{2|\boldsymbol{h}_j^m|}\nonumber\\
&\geq\sum_{j=1}^N|\boldsymbol{h}_j^{m+1}|\sqrt{
\left(\left(\boldsymbol{n}_j^{m+1}\right)^\perp\right)^T
\boldsymbol{Z}_k(\boldsymbol{n}_j^m)\left(\boldsymbol{n}_j^{m+1}
\right)^\perp\gamma(\boldsymbol{n}_j^m)}\nonumber\\
&\geq\sum_{j=1}^N|\boldsymbol{h}_j^{m+1}|\sqrt{\frac{
\gamma^2(\boldsymbol{n}^{m+1}_j)}{\gamma(\boldsymbol{n}^m_j)}
\gamma(\boldsymbol{n}^m_j)}=\sum_{j=1}^N|\boldsymbol{h}_j^{m+1}|
\gamma(\boldsymbol{n}^{m+1}_j)=\int_{\Gamma^{m+1}}
\gamma(\boldsymbol{n}^{m+1})ds.
\end{align}
Taking $\varphi^h=\mu^{m+1}$ in
\eqref{eqn:aniso SP-PFEM eq1}
and $\boldsymbol{\omega}^h=\boldsymbol{X}^{m+1}-\boldsymbol{X}^m$ in
\eqref{eqn:aniso SP-PFEM eq2} and combining the inequality
\eqref{difference between two time steps}, we get
\begin{align}\label{energy disspation gamma(n) final}
W^{m+1}_c-W^m_c&=\int_{\Gamma^{m+1}}\gamma(\boldsymbol{n}^{m+1})ds
-\int_{\Gamma^m}\gamma(\boldsymbol{n}^{m})ds\nonumber\\
&\leq\Bigl(\boldsymbol{Z}_k(\boldsymbol{n}^m)\partial_s \boldsymbol{X}^{m+1},~\partial_s\boldsymbol{X}^{m+1}-
\partial_s\boldsymbol{X}^m\Bigr)_{\Gamma^m}^h\nonumber\\
&=-\tau\Bigl(\partial_s\mu^{m+1},
~\partial_s\mu^{m+1}\Bigr)_{\Gamma^m}^h\leq 0,\quad\forall m\ge0,
\end{align}
which implies the energy dissipation \eqref{engdpfd} for
the SP-PFEM \eqref{eqn:aniso SP-PFEM}.
\end{proof}

Combining Theorems \ref{energy dissipation condition and k_0} and
\ref{energy dissipation condition, theorem}, finally we have

\begin{corollary}[energy dissipation]
\label{Energy dissipation, n theorem}
Assume $\gamma(\boldsymbol{n})$ satisfies \eqref{engstabgmp} and taking $k(\boldsymbol{n})\ge k_0(\boldsymbol{n})$ in
\eqref{def of the energy matrix Z_k(n)},
then the SP-PFEM \eqref{eqn:aniso SP-PFEM} is unconditionally energy stable.
\end{corollary}

\subsection{Explicit formulas for the minimal stabilizing function}
Here we give explicit formulas of the minimal stabilizing function
$k_0(\boldsymbol{n})$ for several popular anisotropic surface energies
$\gamma(\boldsymbol{n})$ in applications. Denote
\[\boldsymbol{J}=\begin{pmatrix}0&-1\\1&0\end{pmatrix},
\qquad \boldsymbol{Z}_0(\boldsymbol{n})=\begin{pmatrix}1&n_1n_2\\n_1n_2&1\end{pmatrix},
\qquad \forall \,\boldsymbol{n}=\begin{pmatrix}n_1\\ n_2\end{pmatrix}\in
{\mathbb S}^1. \]

\begin{lemma}[\color{black}Riemannian-like \color{black}  metric]\label{lem: Riemannian}
When $\gamma(\boldsymbol{n})$ is taken as the \color{black}Riemannian-like \color{black}  metric anisotropic surface energy \eqref{ellipsoidal},  we have $k_0(\boldsymbol{n})\leq k_1(\boldsymbol{n})$, where
\begin{equation}\label{k0nRm11}
k_1(\boldsymbol{n})=\sum_{l=1}^L\gamma_l(\boldsymbol{n})^{-1}
\text{Tr}(\boldsymbol{G}_l),\qquad
\boldsymbol{Z}_{k_1}(\boldsymbol{n})=
\sum_{l=1}^L\gamma_l(\boldsymbol{n})^{-1}\boldsymbol{J}^T
\boldsymbol{G}_l\boldsymbol{J}, \qquad
\forall
\boldsymbol{n}\in \mathbb{S}^1,
\end{equation}
and $k_0(\boldsymbol{n})=k_1(\boldsymbol{n})$ if $L=1$.
\end{lemma}

\color{black}
The proof can be found in Appendix B.
\color{black}

\begin{remark}
By taking $k(\boldsymbol{n})=k_1(\boldsymbol{n})$ in \eqref{def of the energy matrix Z_k(n)} and using the semi-implicit discretization $\boldsymbol{n}^m$ instead of $\boldsymbol{n}^{m+\frac{1}{2}}$, the SP-PFEM \eqref{eqn:aniso SP-PFEM} collapses to the BGN formulation used in \cite{barrett2008numerical}.
\end{remark}

\begin{lemma}[$l^r$-norm metric]\label{lem: lr} When $\gamma(\boldsymbol{n})$ is taken as
the $l^r$-norm metric anisotropic surface energy \eqref{lrnormase}, we have

(i) when $r=4$, $k_0(\boldsymbol{n})=2\gamma(\boldsymbol{n})^{-3}$ and
$\boldsymbol{Z}_{k_0}(\boldsymbol{n})=\gamma(\boldsymbol{n})^{-3}\boldsymbol{Z}_0(\boldsymbol{n})$, and

(ii) when $r=6$, $k_0(\boldsymbol{n})=
2\gamma(\boldsymbol{n})^{-5}(n_1^4+n_1^2n_2^2+n_2^4)$.
\end{lemma}

\color{black}
The proof can be found in Appendix C.

\begin{lemma}[$m$-fold]\label{lem: 2fold} When $\gamma(\boldsymbol{n})$ is taken as the $m$-fold anisotropy \eqref{kfold},  we have

(i) when $\gamma(\boldsymbol{n})=1+\beta \cos 2\theta$, then
\begin{equation}\label{k02fold1}
  k_0(\boldsymbol{n})=4-2\gamma(\boldsymbol{n})+\frac{4\beta^2}
  {\gamma(\boldsymbol{n})}; and
\end{equation}

(ii) when $\gamma(\boldsymbol{n})=1+\beta \cos 4\theta$, then
\begin{equation}\label{k02fold2}
  k_0(\boldsymbol{n})\leq 2\gamma(\boldsymbol{n})+\frac{16\beta+16\beta^2}{\gamma(\boldsymbol{n})}:=
  k_1(\boldsymbol{n}).
\end{equation}
\end{lemma}

The proof can be found in Appendix D.
\color{black}

\section{Numerical results}

In this section, we numerically implement the SP-PFEM \eqref{eqn:aniso SP-PFEM}
for simulating the evolution of closed curves under anisotropic surface diffusion. Numerical results demonstrate the high performance of the proposed scheme, e.g., the spatial/temporal convergence rates, energy dissipation, area conservation, and asymptotic quasi-uniform mesh distribution.
Here, the distance between two closed curves $\Gamma_1$ and $\Gamma_2$ is measured by the manifold distance $M(\Gamma_1,\Gamma_2)$
which was introduced in the reference \cite{bao2020energy}.

Since formally the scheme is first-order accurate in time and second-order accurate in space, the mesh size $h$ and the time step $\tau$ are chosen as $\tau=\mathcal{O}(h^2)$, e.g. $\tau=h^2$, except where noted.
Let $\Gamma^{m}$ be the numerical approximation of $\Gamma^h(t=t_m=m\tau)$
with  mesh size $h$ and time step $\tau$, the numerical error is then measured as
\begin{equation}
e^{h}(t_m):=M(\Gamma^{m},\Gamma(t=t_m)),\qquad m\ge0.
\end{equation}
Because the exact solution can not be obtained analytically, we choose fine meshes $h=h_e$, $\tau=\tau_e$ to obtain $\Gamma(t=t_m)$ numerically, e.g.
$h_e=2^{-8}$ and $\tau_e=2^{-16}$.

The normalized area loss and the mesh ratio $R^h(t_m)$, which
indicates the mesh quality during evolution,
are defined as
\begin{eqnarray}
\qquad \frac{\Delta A^h(t_m)}{A^h(0)}:=\frac{A^h(t_m)-A^h(0)}{A^h(0)},\quad
R^h(t_m):=\frac{\max_{1\le j\le N}\ |\boldsymbol{h}_j^m|}{\min_{1\le j\le N}\ |\boldsymbol{h}_j^m|},\quad m\ge0,
\end{eqnarray}
where $A^h(t_m)$ is the area of the inner region enclosed by
$\Gamma^{m}$. 

In the following simulations, the initial shape in \eqref{init}
is always chosen as an ellipse with length $4$ and width $1$ except where noted,
and the tolerance of the Newton iteration in the SP-PFEM \eqref{eqn:aniso SP-PFEM} is chosen as $10^{-12}$.

\subsection{Convergence rates and energy dissipation} In order to test convergence rates of the SP-PFEM \eqref{eqn:aniso SP-PFEM}, without loss of generality,
we choose the following two kinds of anisotropic surface energies:

\begin{itemize}
\item Case I: the \color{black}Riemannian-like \color{black}  metric anisotropic surface energy \eqref{ellipsoidal} with $L=1$ and $\boldsymbol{G}_1={\rm diag}(1,2):=\boldsymbol{G}$, and the corresponding minimal stabilizing function $k_0(\boldsymbol{n})$ is given explicitly in \eqref{k0nRm11};

\item Case II: the $l^r$-norm metric anisotropic surface energy \eqref{lrnormase} with $r=4$ and the corresponding minimal stabilizing function $k_0(\boldsymbol{n})$ is given explicitly in Lemma \ref{lem: lr}.
\end{itemize}

\begin{figure}[htp!]
\centering
\includegraphics[width=1\textwidth]{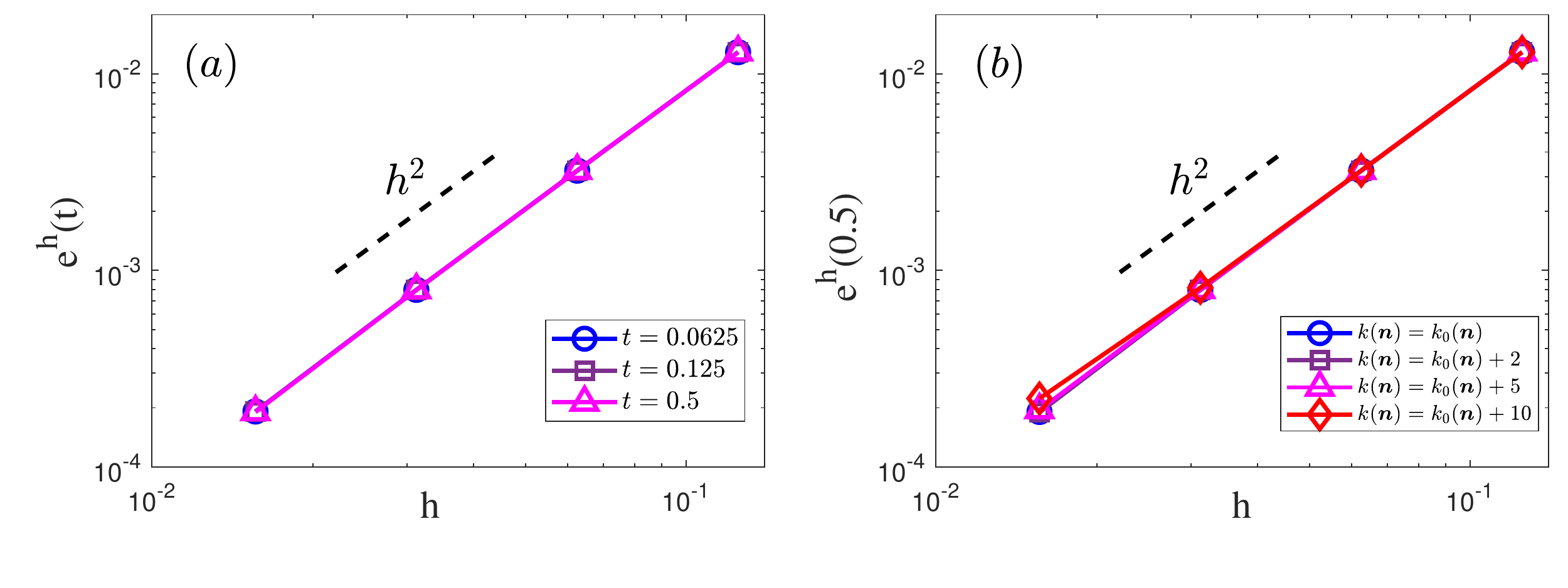}
\includegraphics[width=1\textwidth]{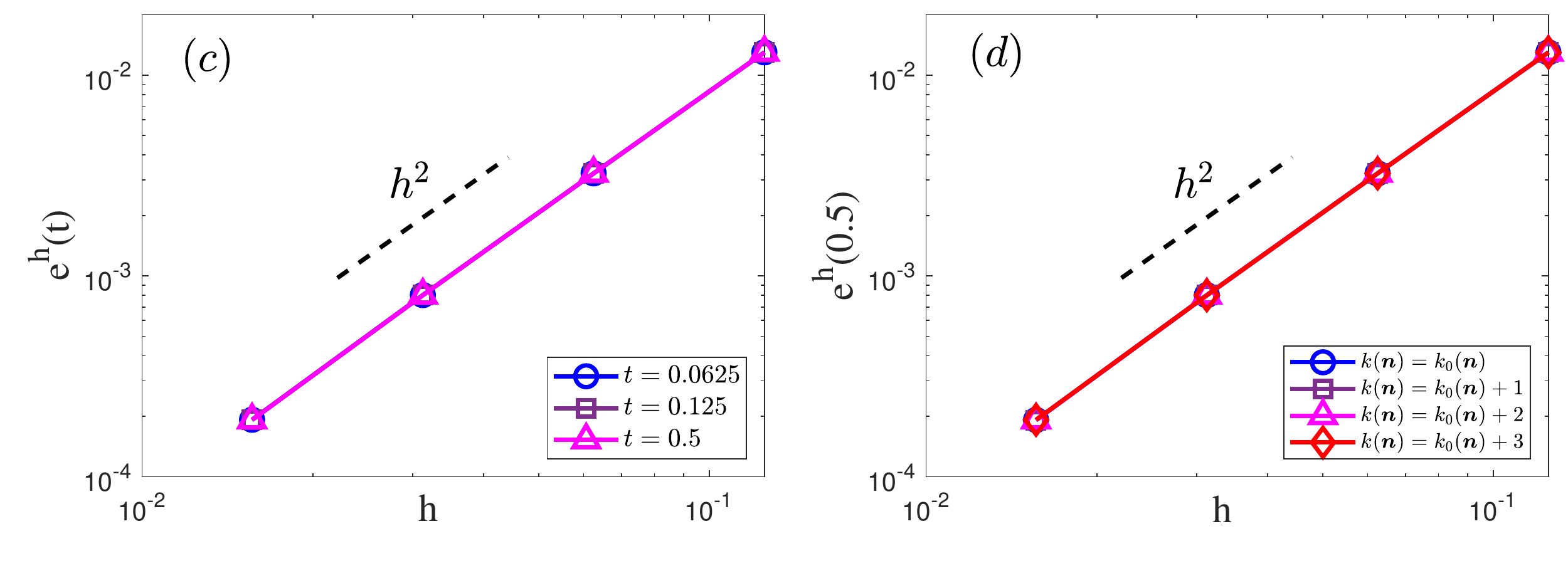}
\caption{Spatial convergence rates of the SP-PFEM \eqref{eqn:aniso SP-PFEM}
for: Case I at different times with $k(\boldsymbol{n})=k_0(\boldsymbol{n})$ in \eqref{k0nRm11} (a), and
 at time $t=0.5$ for different $k(\boldsymbol{n})$ (b); and Case II at different times with $k(\boldsymbol{n})=k_0(\boldsymbol{n})$ in Lemma \ref{lem: lr} (c), and
at time $t=0.5$ for different $k(\boldsymbol{n})$ (d).}
\label{fig:convergence rate SP-PFEM}
\end{figure}

Fig.~\ref{fig:convergence rate SP-PFEM} plots spatial convergence
rates of the SP-PFEM at different times under a fixed value
$k(\boldsymbol{n})$ in \eqref{def of the energy matrix Z_k(n)}
or different values of $k(\boldsymbol{n})$ under a fixed time $t=0.5$. Fig. \ref{fig:area and energy SP-PFEM} depicts time evolution of the normalized area loss and the normalized energy under different parameters. Fig.~\ref{fig:mesh ratio SP-PFEM} depicts
time evolution of the mesh ratio $R^h(t)$ under different mesh sizes $h$, time
steps $\tau$ and $k(\boldsymbol{n})$ for the above two cases.

\begin{figure}[htp!]
\centering
\includegraphics[width=0.49\textwidth]{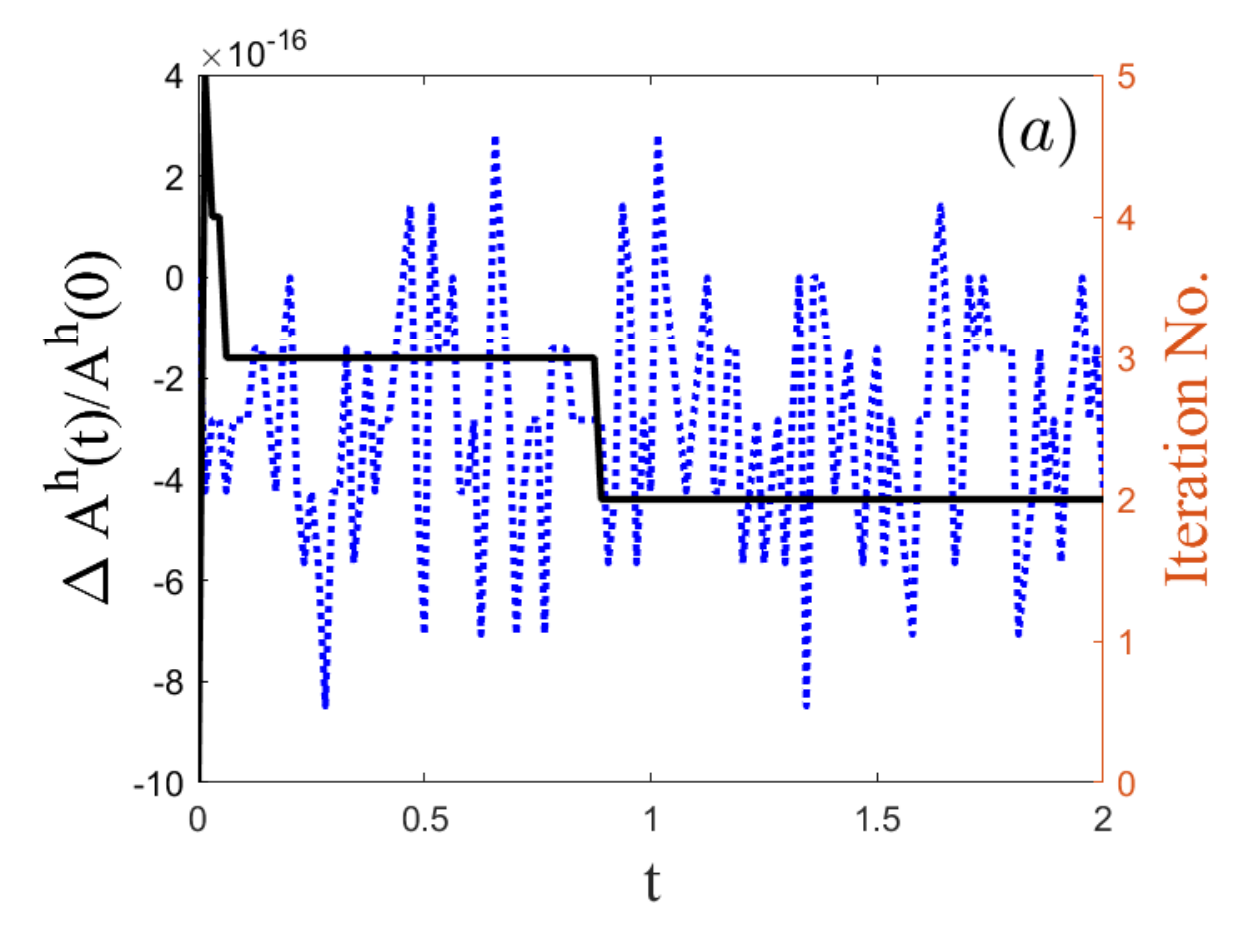}
\includegraphics[width=0.49\textwidth]{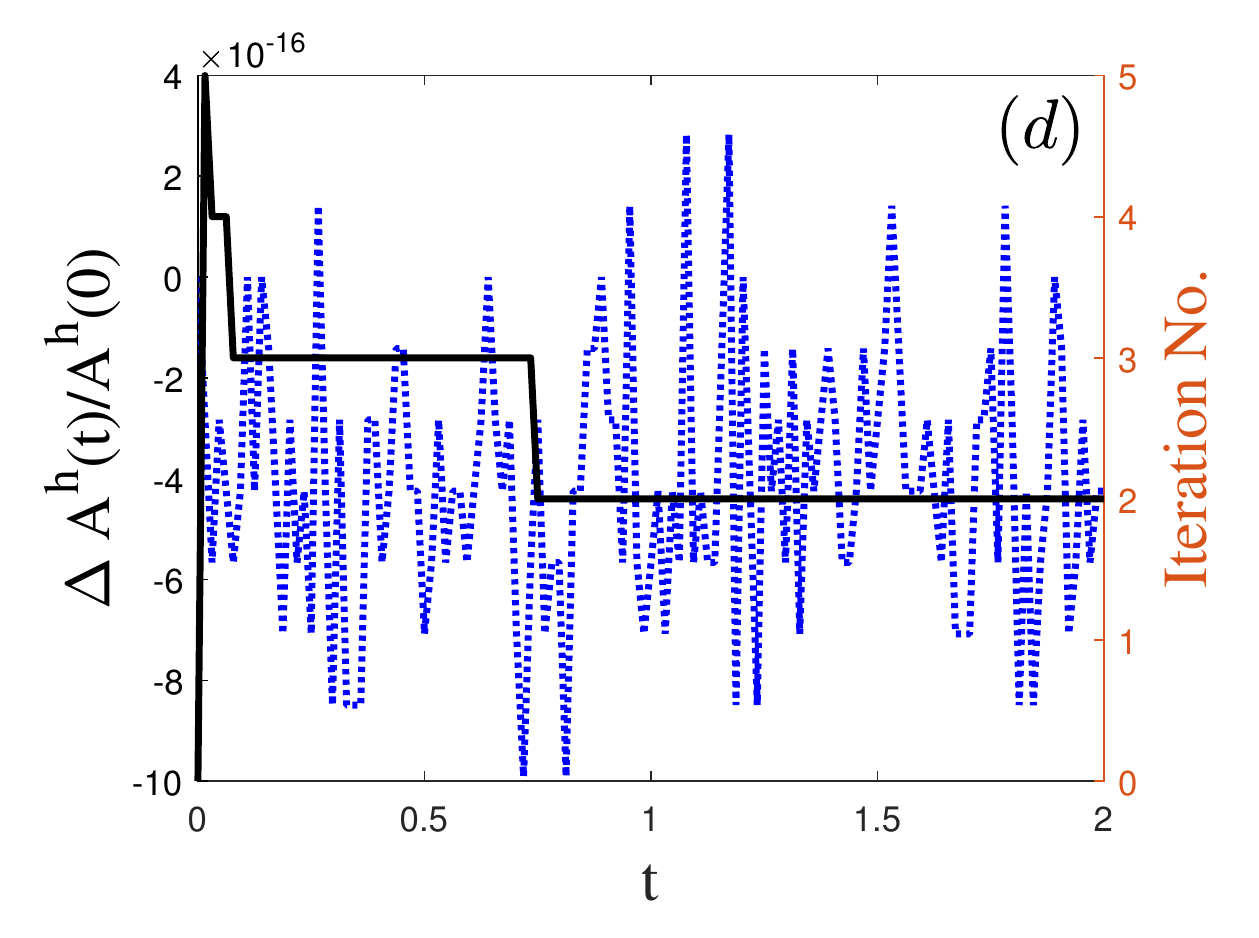}
\includegraphics[width=0.49\textwidth]{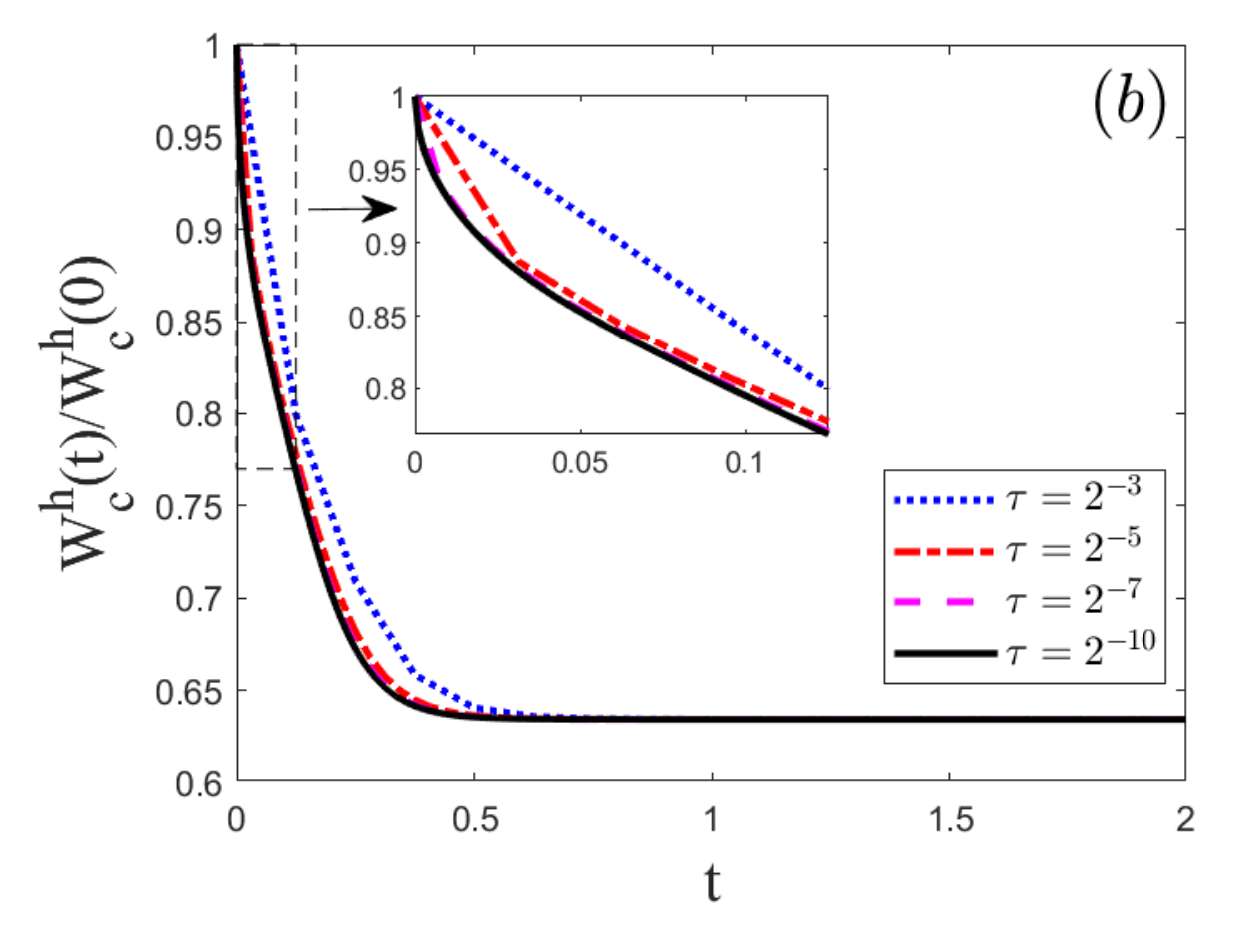}
\includegraphics[width=0.49\textwidth]{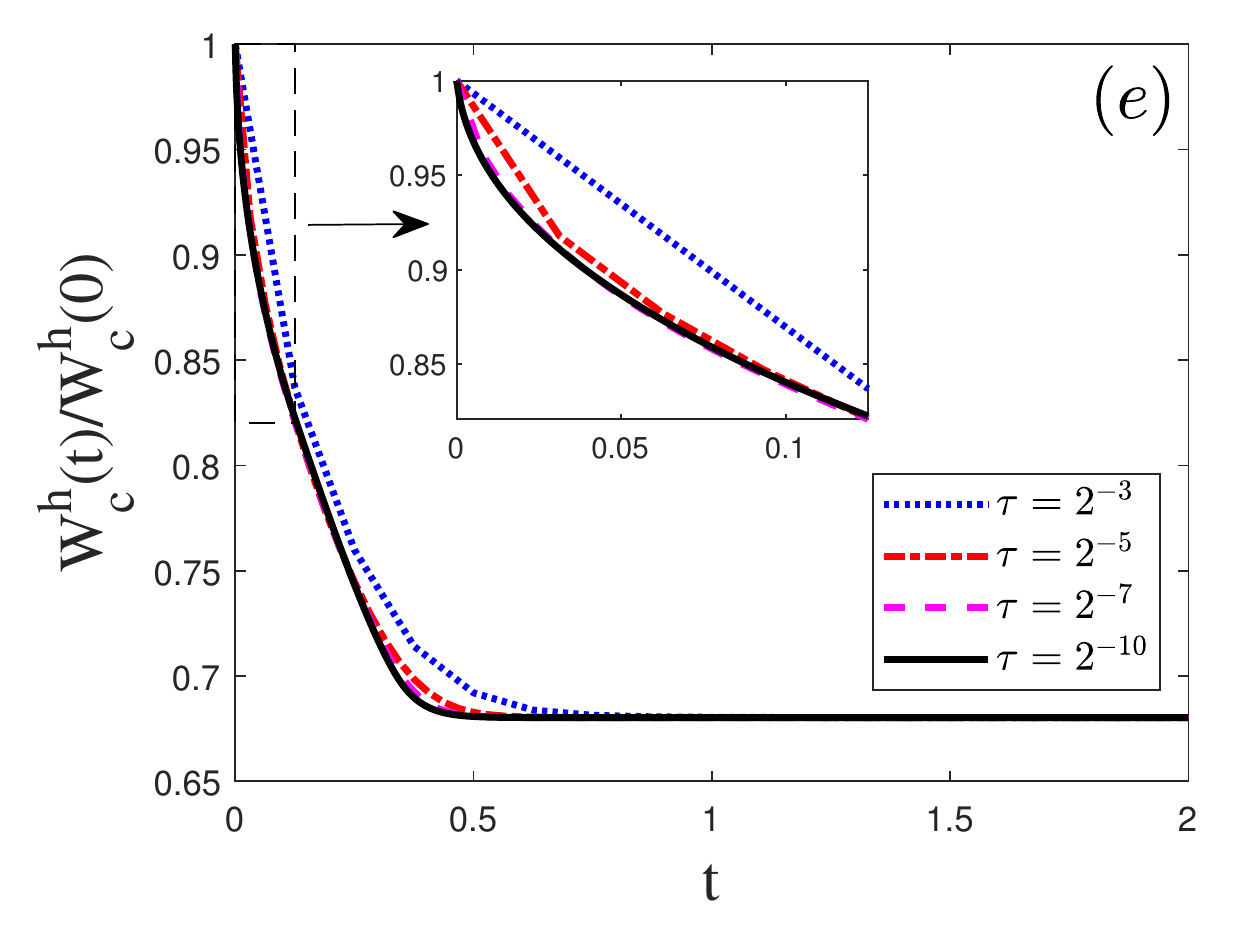}
\includegraphics[width=0.49\textwidth]{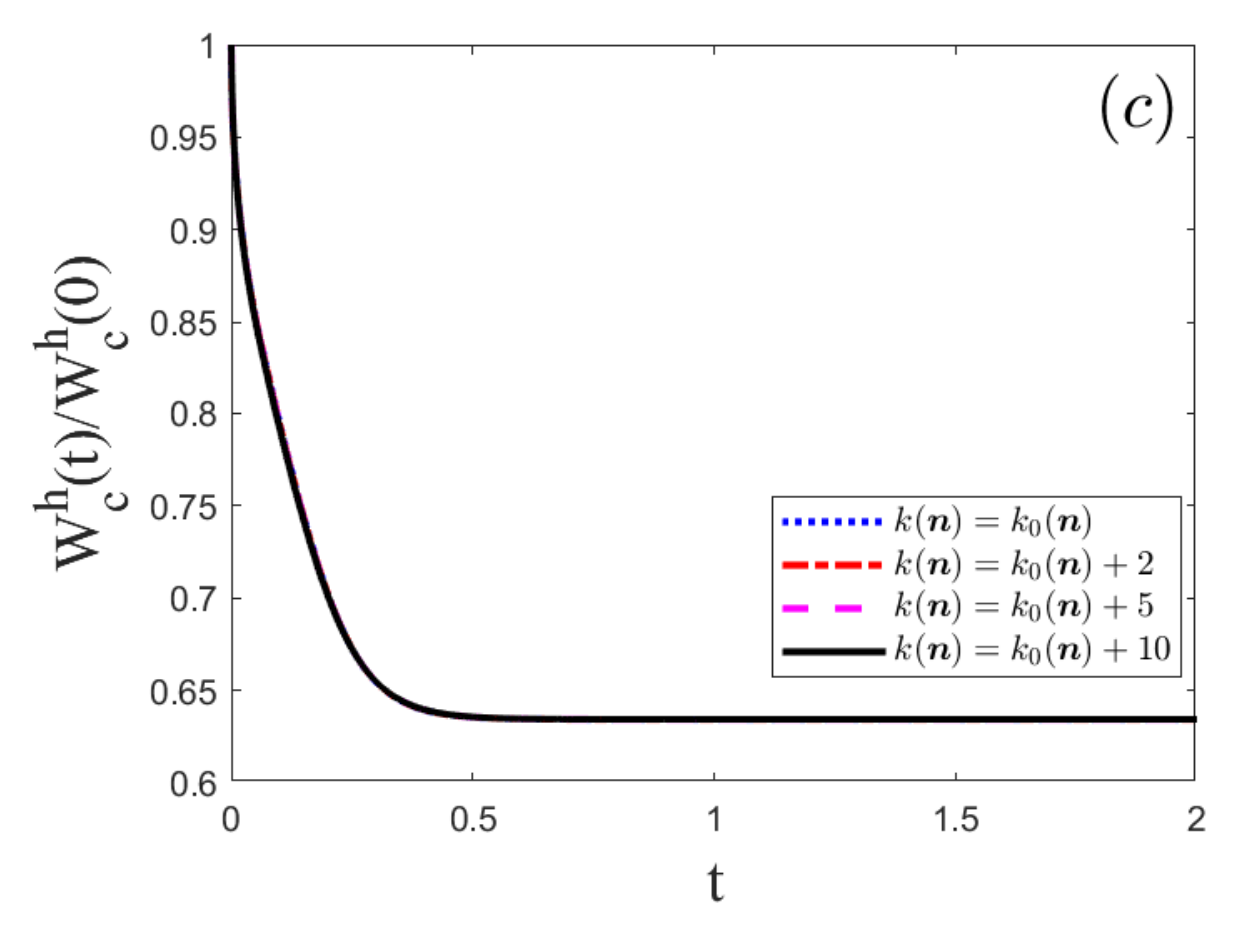}
\includegraphics[width=0.49\textwidth]{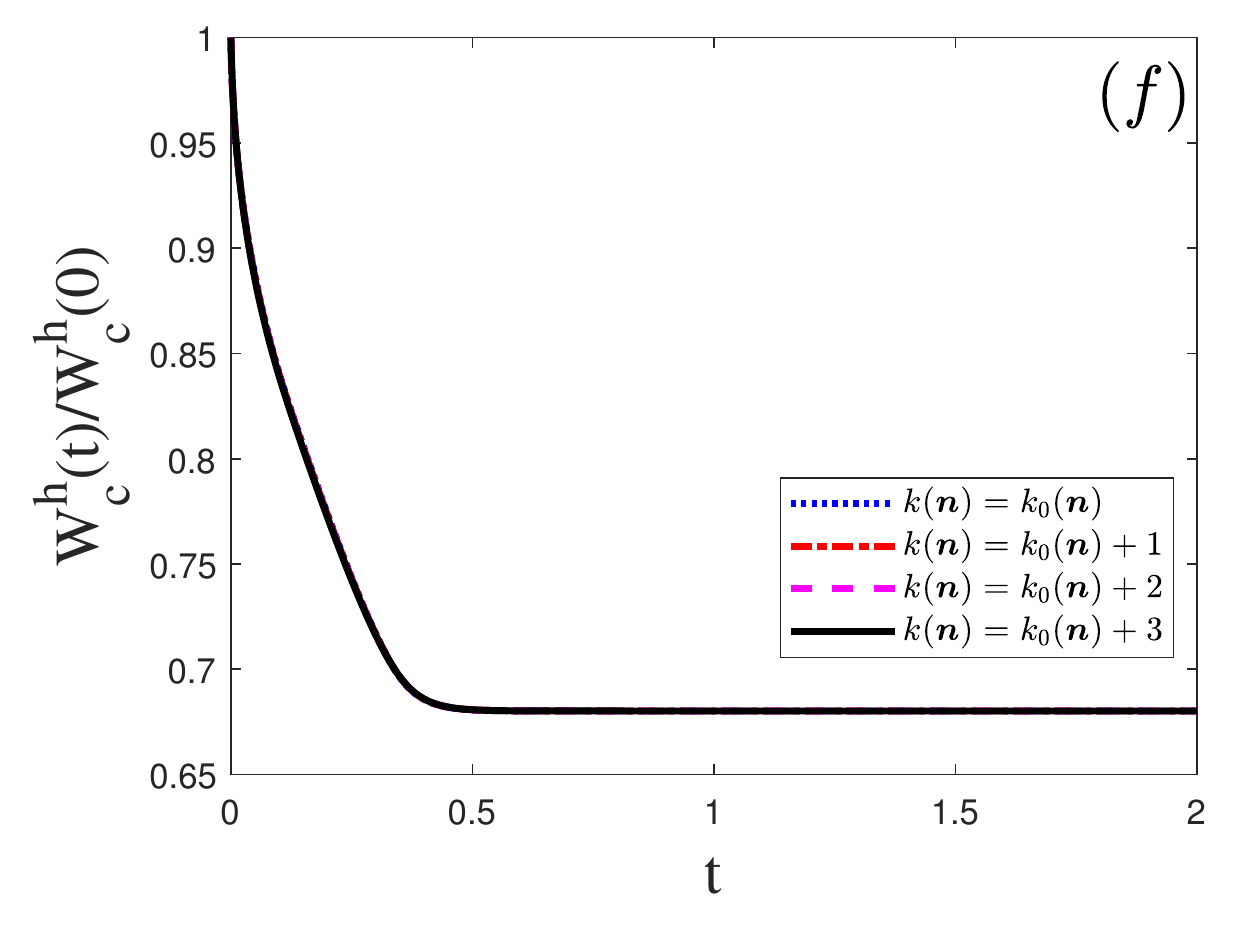}
\caption{Time evolution of the normalized area loss $\frac{\Delta A^h(t)}{A^h(0)}$ (first row, blue dashed line) and iteration number (first row, black line) and the normalized energy $\frac{W^h_c(t)}{W^h_c(0)}$ (second and third rows) for:
Case I with $k(\boldsymbol{n})=k_0(\boldsymbol{n})$ in \eqref{k0nRm11}
for $h=2^{-3}$ (a), and with $h=2^{-3}$ for different $\tau$ (b), and with $h=2^{-3}$ for different
$k(\boldsymbol{n})$ (c);  and Case II with $k(\boldsymbol{n})=k_0(\boldsymbol{n})$ in Lemma \ref{lem: lr} for $h=2^{-3}$ (d), and with $h=2^{-3}$  for different $\tau$ (e), and with $h=2^{-3}$ for different
$k(\boldsymbol{n})$ (f).}
\label{fig:area and energy SP-PFEM}
\end{figure}

\begin{figure}[htp!]
\centering
\includegraphics[width=0.95\textwidth]{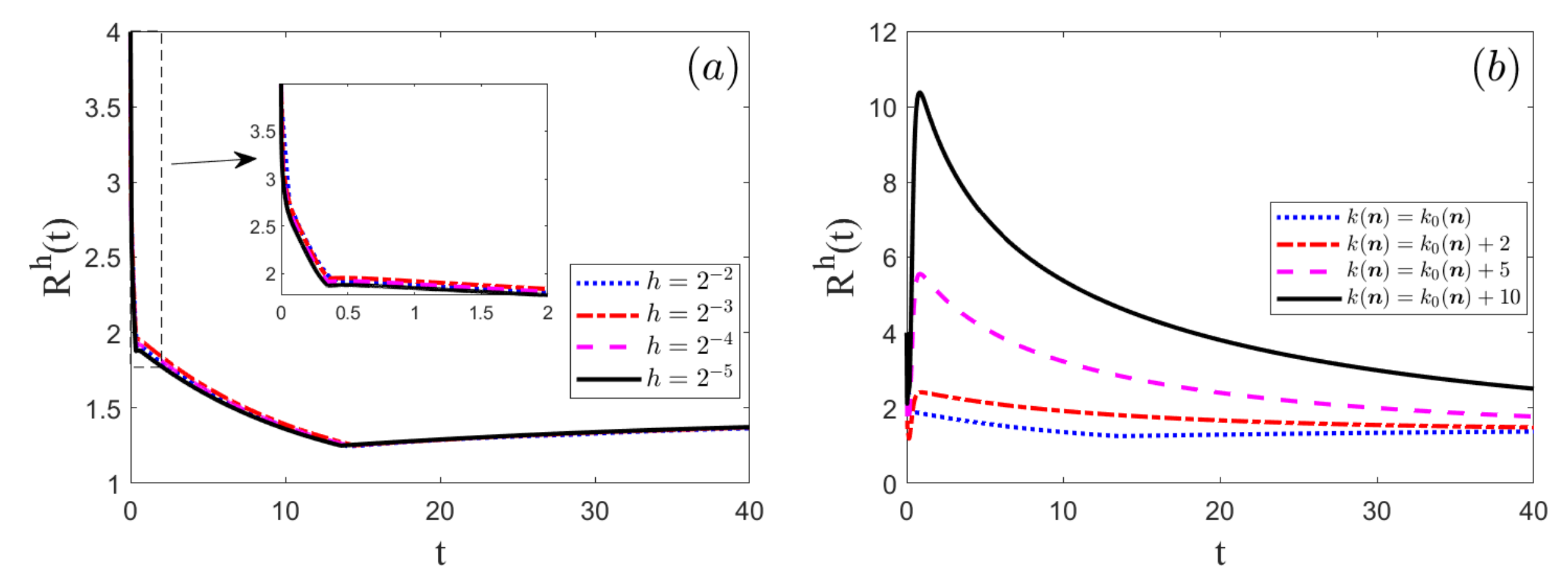}
\includegraphics[width=0.95\textwidth]{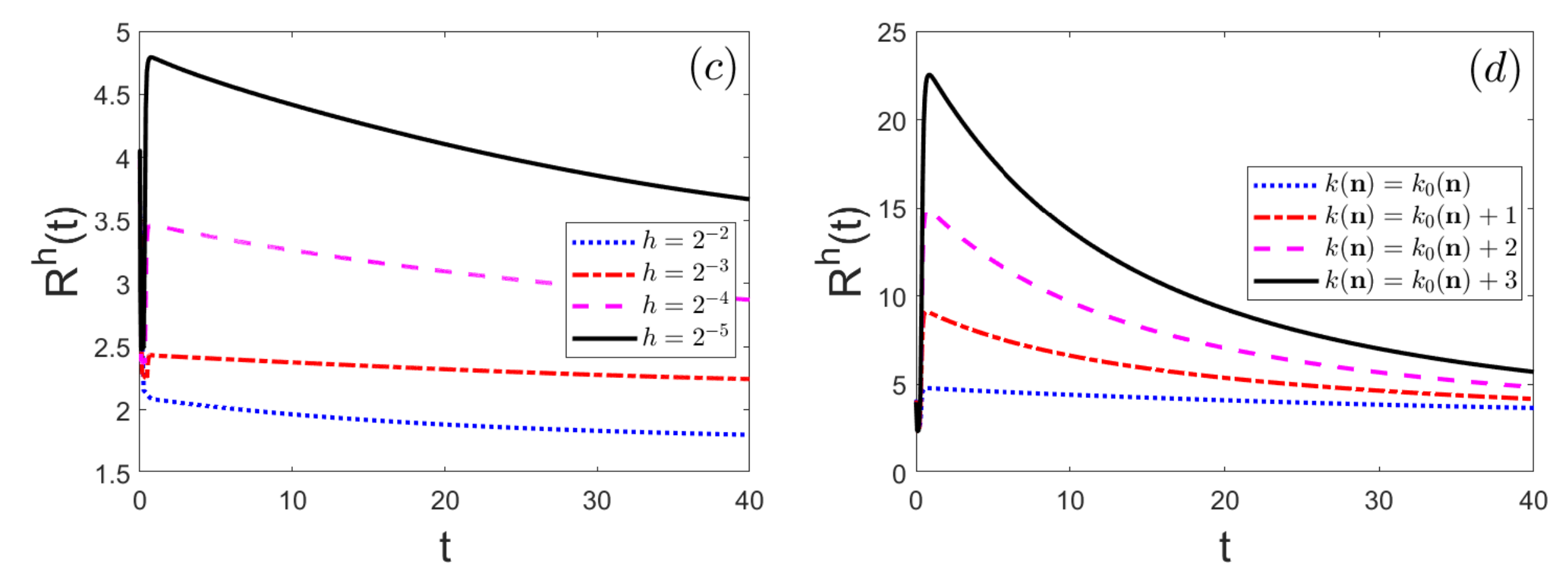}
\caption{Time evolution of the mesh ratio $R^h(t)$ for:
Case I with $k(\boldsymbol{n})=k_0(\boldsymbol{n})$ in \eqref{k0nRm11} for
different $h$ (a), and with
$h=2^{-5}$ for different $k(\boldsymbol{n})$ (b);  and Case II with $k(\boldsymbol{n})=k_0(\boldsymbol{n})$ in Lemma \ref{lem: lr} for
different $h$ (c), and with
$h=2^{-5}$ for different $k(\boldsymbol{n})$ (d).}
\label{fig:mesh ratio SP-PFEM}
\end{figure}

From Figs. \ref{fig:convergence rate SP-PFEM}--\ref{fig:mesh ratio SP-PFEM},
we can obtain the following results for the SP-PFEM
\eqref{eqn:aniso SP-PFEM} for simulating anisotropic surface diffusion of closed curves:

(i) The SP-PFEM is second-order accurate in space (cf. Fig.~\ref{fig:convergence rate SP-PFEM});

(ii) The area is conserved numerically up to the round-off error around $10^{-16}$ (cf. Fig.~\ref{fig:area and energy SP-PFEM}(a)\&(d));

(iii) The number of Newton iteration at each time step is around 2 to 4, thus it is very efficient (cf. Fig.~\ref{fig:area and energy SP-PFEM}(a)\&(d));

(iv) The SP-PFEM is unconditionally energy-stable when $k(\boldsymbol{n})$ satisfies the energy dissipation condition in Theorem 4.5
(cf. Fig.~\ref{fig:area and energy SP-PFEM}(b)-(c)\&(e)-(f));

(v) The mesh ratio $R^h(t=t_m)$ approaches a constant $C$ when $t\gg1$ for each case, which indicates asymptotic quasi-uniform mesh distribution,
no matter what kind of anisotropic surface energy is used as long as it is weakly anisotropic.

\subsection{Application for morphological evolutions}
Here, we use the SP-PFEM \eqref{eqn:aniso SP-PFEM} to simulate the morphological evolution under different anisotropic surface
energies, i.e.,
morphological evolutions of closed curves from a $4\times 1$ rectangle towards their corresponding equilibrium shapes.
Fig. \ref{fig:shape closed} depicts morphological evolutions for the four different weakly anisotropic surface energies including (a) the regularized $l^1$-norm metric \eqref{l1normase}  with $\varepsilon=0.1$ by taking $k(\boldsymbol{n})=k_1(\boldsymbol{n}):=\frac{1.01}{\sqrt{n_1^2+0.01n_2^2}}
+\frac{1.01}{\sqrt{0.01n_1^2+n_2^2}}$, (b) the $l^4$-norm metric
\eqref{lrnormase} with $r=4$ and  $k(\boldsymbol{n})=k_0(\boldsymbol{n})$
given in Lemma \ref{lem: lr}, (c) $2$-fold anisotropic energy \eqref{kfold}
with $m=2$, $\theta_0=\frac{\pi}{2}$ and $\beta=\frac{1}{3}$ and $k(\boldsymbol{n})=k_0(\boldsymbol{n})$ given in \eqref{k02fold1}, and (d) the Riemannian-like metric \eqref{ellipsoidal} with $L=1$ and $\boldsymbol{G}_1={\rm diag}(1,2):=\boldsymbol{G}$ and  $k(\boldsymbol{n})=k_0(\boldsymbol{n})$
given in \eqref{k0nRm11}.
Figs. \ref{fig:shape closed strongly anisotropy 1} and
\ref{fig:shape closed strongly anisotropy 2} show morphological evolutions and the normalized energy $\frac{W_c^h(t)}{W_c^h(0)}$ under the $2$-fold $\gamma(\boldsymbol{n})=1+\frac{3}{5}\cos(2\theta)$ and the $4$-fold $\gamma(\boldsymbol{n})=1+\frac{3}{10}\cos(4\theta)$, with $k(\boldsymbol{n})$ given in \eqref{k02fold1}, \eqref{k02fold2}, respectively, which are both strongly anisotropic surface energies. The Frank diagrams of the above anisotropic energies are all shown in Fig.~\ref{fig:Frank diagram}.

\begin{figure}[htp!]
\centering
\includegraphics[width=0.95\textwidth]{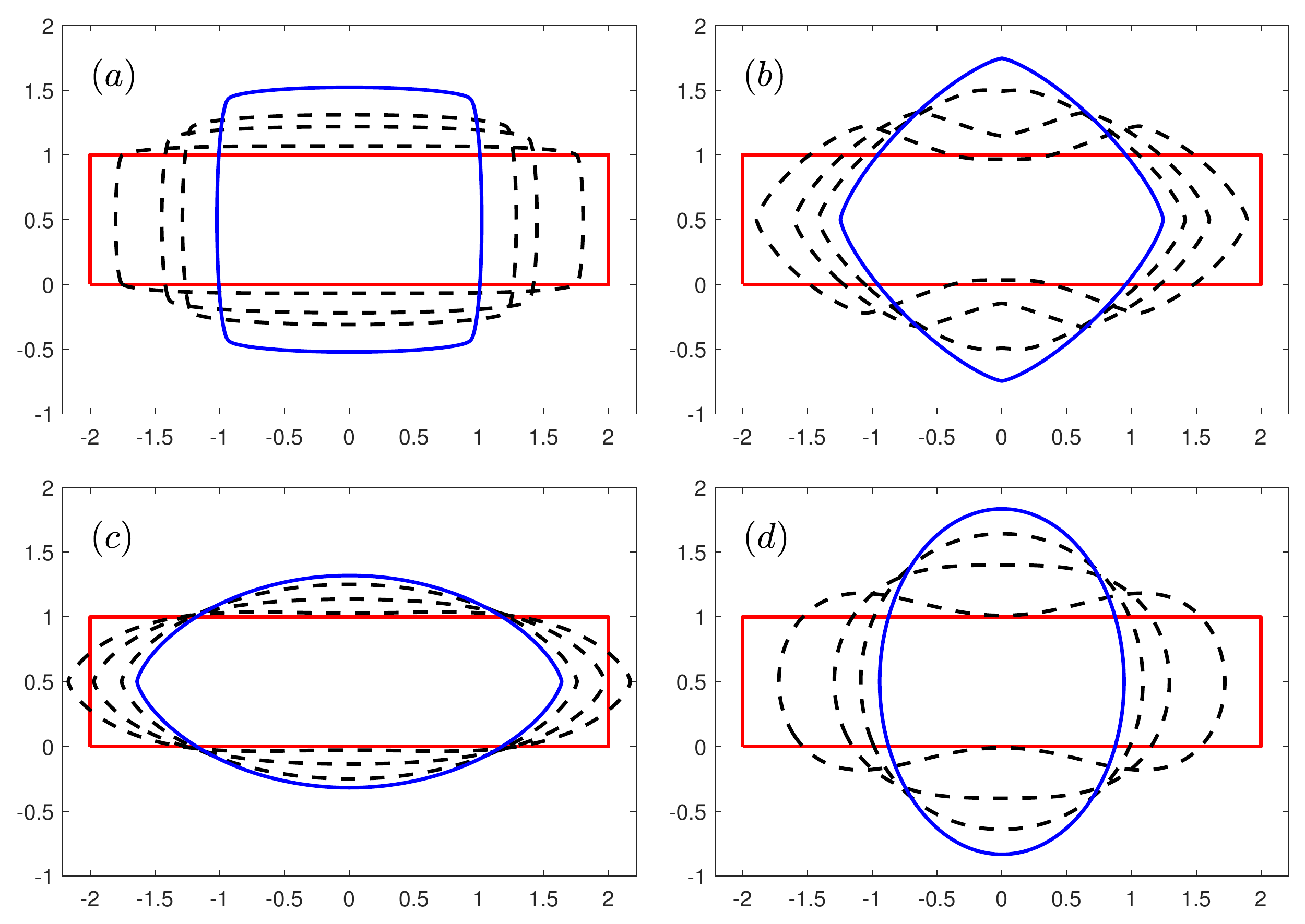}
\caption{Morphological evolutions of a close rectangular curve under anisotropic surface diffusion with different anisotropic surface energies:  (a) regularized $l^1$-norm metric $\gamma(\boldsymbol{n})=\sqrt{n_1^2+0.01n_2^2}+\sqrt{0.01n_1^2+n_2^2}$; (b) $l^4$-norm metric $\gamma(\boldsymbol{n})=\sqrt[4]{n_1^4+n_2^4}$; (c) $2$-fold $\gamma(\boldsymbol{n})=1+\frac{1}{3}\cos(2(\theta-\frac{\pi}{2}))$; and (d) Riemannian-like metric $\gamma(\boldsymbol{n})=\sqrt{\boldsymbol{n}^T \begin{pmatrix}
1&0 \\ 0&2 \end{pmatrix} \boldsymbol{n}}$,
where the parameters $h=2^{-6}, \tau=h^2$, and the red line, black dashed line and blue line represent the initial shape, intermediate shape and equilibrium shape, respectively.}
\label{fig:shape closed}
\end{figure}

\begin{figure}[htp!]
\centering
\includegraphics[width=0.90\textwidth]{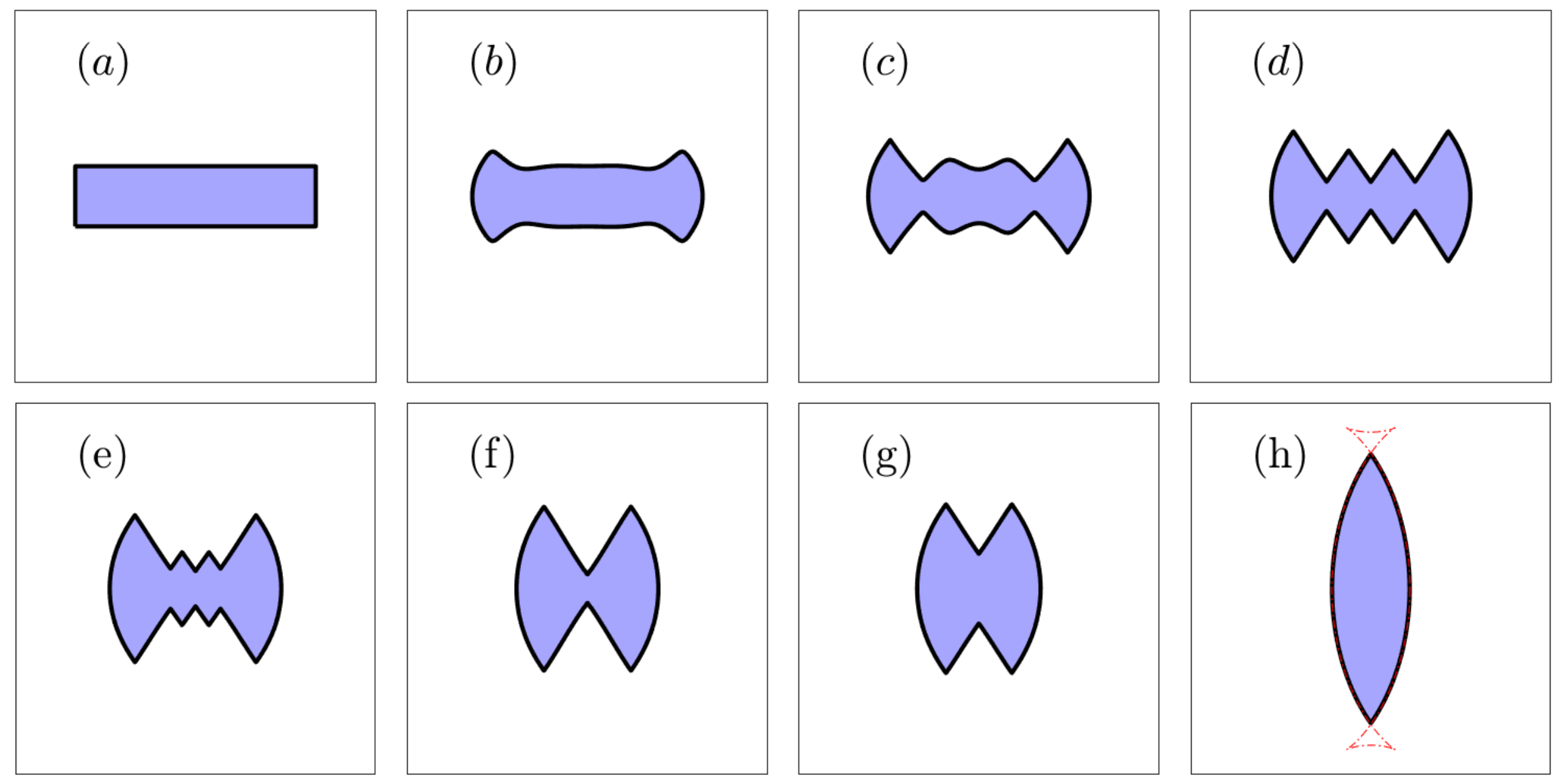}
\includegraphics[width=0.90\textwidth]{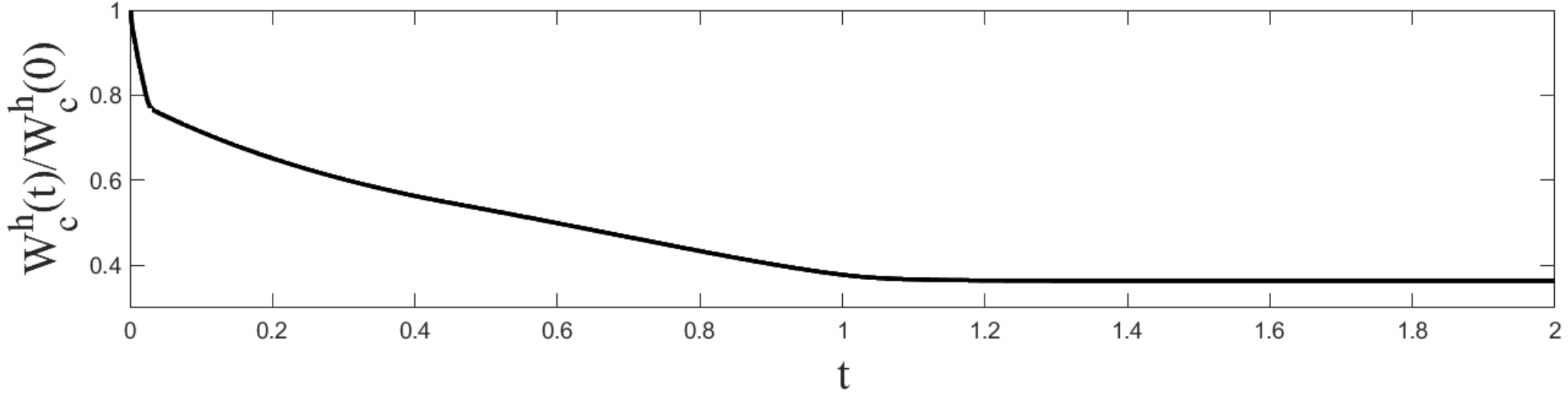}
\caption{Morphological evolutions and the normalized energy of a close rectangular curve under
anisotropic surface diffusion with the strongly $2$-fold anisotropic surface energy $\gamma(\boldsymbol{n})=1+\frac{3}{5}\cos(2\theta)$ towards its equilibrium at different times:  (a) $t=0$; (b) $t=10\tau$; (c) $t=20\tau$; (d) $t=100\tau$; (e) $t=250\tau$; (f) $t=500\tau$; (g) $t=700\tau$; and (h) $t=5000\tau$, where the other parameters are chosen as the same as Fig.~\ref{fig:shape closed}.}
\label{fig:shape closed strongly anisotropy 1}
\end{figure}
\begin{figure}[htp!]
\centering
\includegraphics[width=0.95\textwidth]{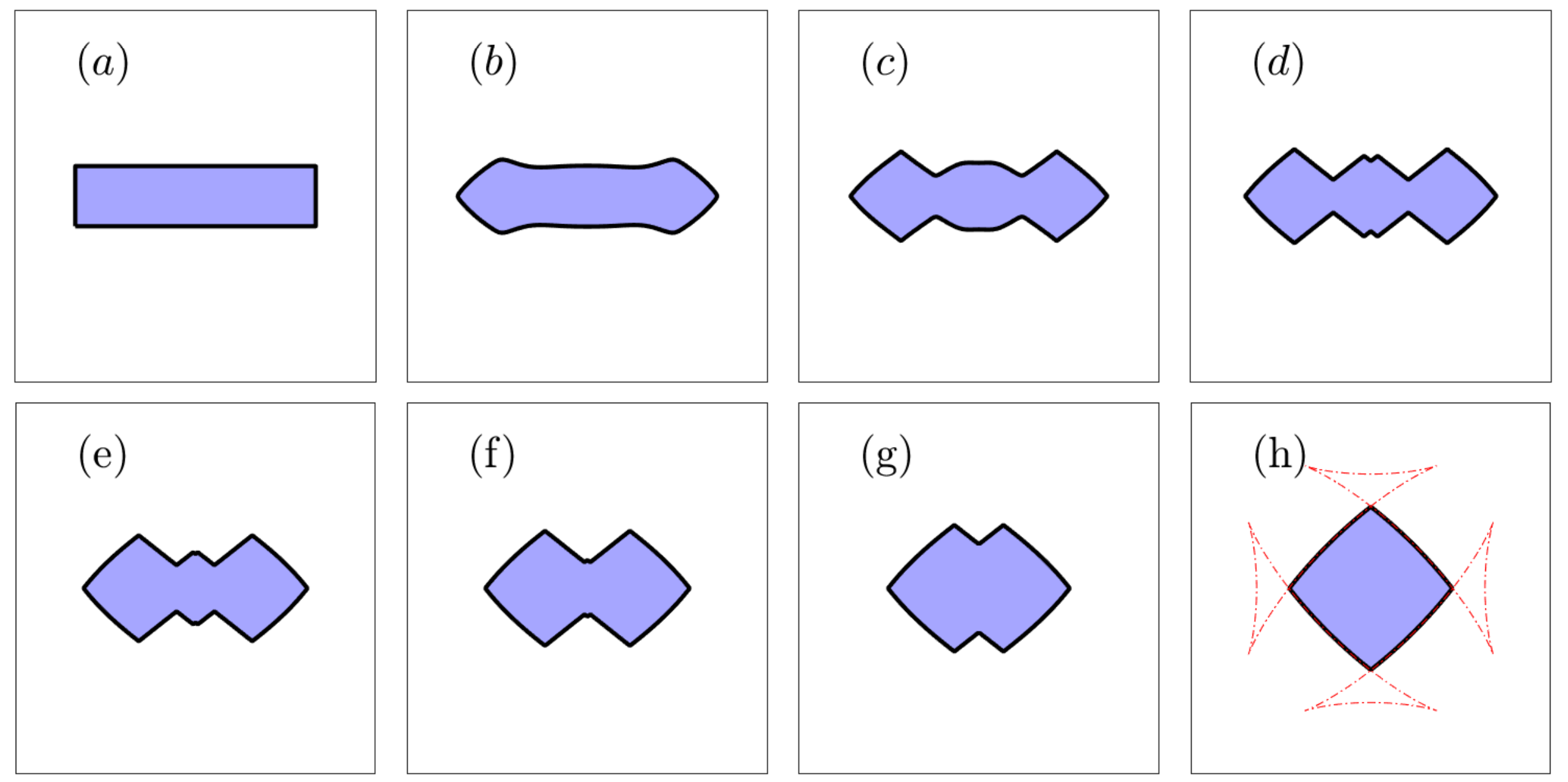}
\includegraphics[width=0.95\textwidth]{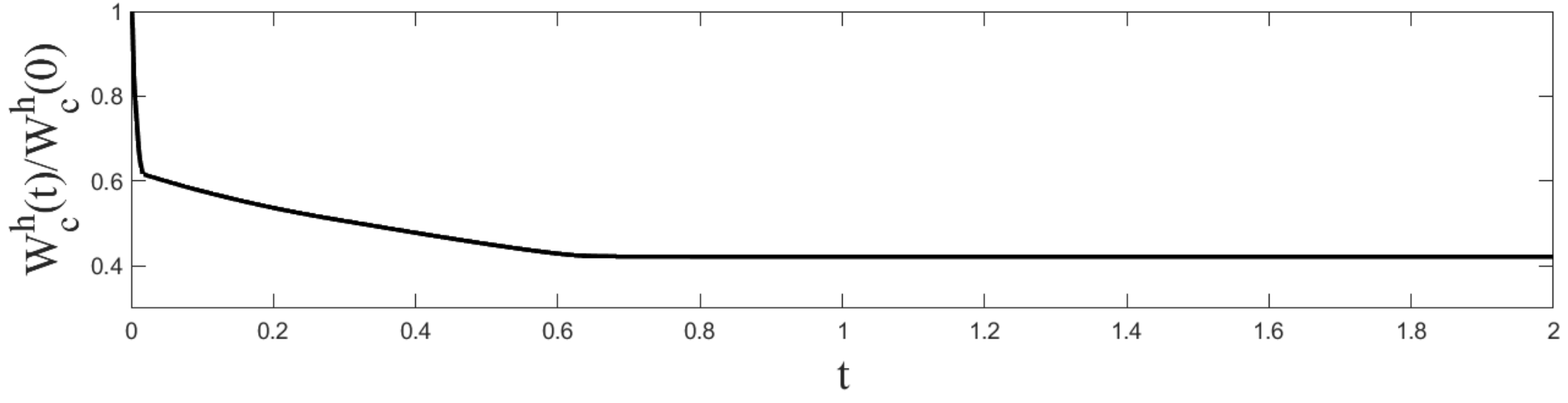}
\caption{Morphological evolutions and the normalized energy of a close rectangular curve under
anisotropic surface diffusion with the strongly $4$-fold anisotropic surface energy $\gamma(\boldsymbol{n})=1+\frac{3}{10}\cos(4\theta)$ towards its equilibrium at different times:  (a) $t=0$; (b) $t=5\tau$; (c) $t=10\tau$; (d) $t=20\tau$; (e) $t=160\tau$; (f) $t=300\tau$; (g) $t=500\tau$; and (h) $t=5000\tau$, where the parameters are chosen as $h=2^{-5}, \tau=h^2$, and the red dashed line in (h) is the Wulff envelope.}
\label{fig:shape closed strongly anisotropy 2}
\end{figure}

\begin{figure}[htp!]
\centering
\includegraphics[width=0.90\textwidth]{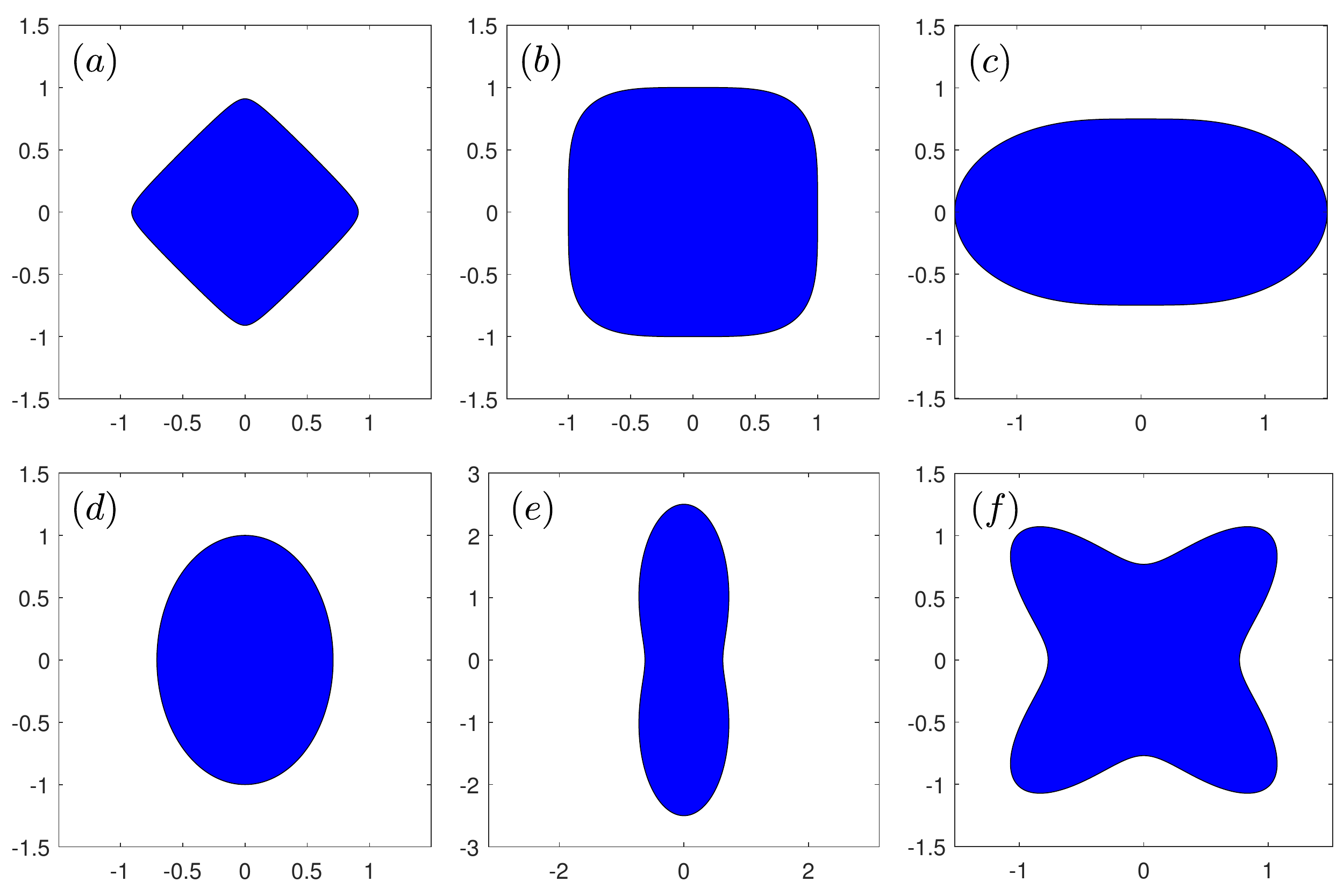}
\caption{\color{black}The Frank diagrams of the weakly anisotropic energies: (a)-(d) used in Fig.~\ref{fig:shape closed} respectively; and the strongly anisotropic energies: (e) $\gamma(\boldsymbol{n})=1+\frac{3}{5}\cos(2\theta)$ in Fig.~\ref{fig:shape closed strongly anisotropy 1}, and (f) $\gamma(\boldsymbol{n})=1+\frac{3}{10}\cos(4\theta)$ in Fig.~\ref{fig:shape closed strongly anisotropy 2}.}
\label{fig:Frank diagram}
\end{figure}

As shown in Fig. \ref{fig:shape closed}(a)--(b), if we choose the anisotropy as the regularized $l^1$-norm metric or the $l^4$-norm metric, the equilibrium shapes are almost ``faceting'' squares; for $2$-fold anisotropy (c.f. \ref{fig:shape closed}(c)),  the number of edges in its equilibrium shape is exactly two; and for the \color{black}Riemannian-like \color{black}  metric anisotropic energy (c.f. \ref{fig:shape closed}(d)),  the equilibrium shape is an ellipse. The numerical results are perfectly consistent with the theoretical predictions by the well-known Wulff construction~\cite{Wulff01,barrett2008numerical,Bao17}.
Because the anisotropic surface diffusion is area preserving during the evolution, we can easily obtain its theoretical equilibrium shape (or Wulff shape) by using the expression in \cite{Bao17,Jiang2016solid}. As shown in Figs. \ref{fig:shape closed strongly anisotropy 1}(h)\&\ref{fig:shape closed strongly anisotropy 2}(h), the numerical equilibrium shapes are again perfectly consistent with the theoretical
predictions by the Wulff construction in the strongly anisotropic cases. Meanwhile, we can clearly see that the normalized energy is monotonically decreasing during the evolution for the strongly anisotropic cases. Furthermore, we observe that the numerical equilibrium has several ``cusps'', which result from the self intersection of the Wulff envelope~\cite{Bao17}.

\section{Conclusions}

By utilizing a symmetric positive definite surface energy matrix $\boldsymbol{Z}_k(\boldsymbol{n})$ and a stabilizing function $k(\boldsymbol{n})$, we reformulated the anisotropic surface diffusion equation with any arbitrary anisotropic surface energy $\gamma(\boldsymbol{n})$ into a novel symmetrized form and derived a new variational formulation.
We discretized the variational problem in space by the PFEM. For temporal discretization, we proposed a
fully implicit SP-PFEM, which can rigorously preserve the total area up to machine precision. Then we rigorously proved that the proposed SP-PFEM is unconditionally energy-stable under a simple and mild condition \eqref{engstabgmp} on the anisotropic surface energy $\gamma(\boldsymbol{n})$. Finally, numerical results demonstrated that the SP-PFEM is second-order accurate in space, first-order in time, unconditionally energy-stable, and enjoys very good mesh quality during the evolution, and no mesh redistribution procedure is needed even for strongly anisotropic cases. Another important contribution is that the new scheme can also work well for the strongly anisotropic cases (shown in Figs. \ref{fig:shape closed strongly anisotropy 1}-\ref{fig:shape closed strongly anisotropy 2}). In the existing literature, a Willmore regularization energy term is often added into the model to deal with the strongly anisotropic cases~\cite{Spencer04,Li2009geo,Hausser05,Torabi09,Jiang2016solid}, but here we only use one unified scheme to tackle the two cases. In the future, we will further explore the high performance of the schemes, especially for the strongly anisotropic cases; and extend the new variational formulation to anisotropic surface diffusion of open/closed surfaces in three dimensions~\cite{Jiang,Zhao}.

\medskip

\begin{center}
{\textbf{Appendix A.}} The Cahn-Hoffman $\boldsymbol{\xi}$-vector for several anisotropic surface energies
\end{center}
\setcounter{equation}{0}
\renewcommand{\theequation}{A.\arabic{equation}}

\medskip
For the \color{black}Riemannian-like \color{black}  metric surface energy \eqref{ellipsoidal}, we have
\begin{align}\label{ellipsoidalaa1}
&\gamma(\boldsymbol{p})=\sum_{l=1}^L\sqrt{\boldsymbol{p}^T\boldsymbol{G}_l\boldsymbol{p}},
\qquad \forall\boldsymbol{p}\in \mathbb{R}^2_*:=\mathbb{R}^2\setminus \{\boldsymbol{0}\},\\
&\boldsymbol{\xi}=\boldsymbol{\xi}(\boldsymbol{n})=\sum_{l=1}^L
 \gamma_l(\boldsymbol{n})^{-1}\,\boldsymbol{G}_l\,\boldsymbol{n}, \qquad
 \lambda({\boldsymbol{n}})=\sum_{l=1}^L\gamma_l(\boldsymbol{n})^{-3}\,\text{det}(\boldsymbol{G}_l)>0,
\end{align}
which indicates the \color{black}Riemannian-like \color{black}  metric anisotropy is always weakly anisotropic.

For the $l^r$-norm ($r\geq 2$) metric anisotropic surface energy \eqref{lrnormase}, we have
\begin{align}
&\gamma(\boldsymbol{p})=\left\|\boldsymbol{p}\right\|_{l^r}=
 \left(|p_1|^r+|p_2|^r\right)^{\frac{1}{r}}, \qquad \forall\boldsymbol{p}=(p_1,p_2)^T\in \mathbb{R}^2_*,\\
&\boldsymbol{\xi}=\boldsymbol{\xi}(\boldsymbol{n})=
 \gamma(\boldsymbol{n})^{1-r}
  \begin{pmatrix}|n_1|^{r-2}n_1\\ |n_2|^{r-2}n_2\end{pmatrix}, \quad \lambda({\boldsymbol{n}})=(r-1)\frac{|n_1n_2|^{r-2}}
{\gamma(\boldsymbol{n})^{2r-1}},\quad  \forall\boldsymbol{n}\in \mathbb{S}^1,\label{eqxx}
\end{align}
which indicates the $l^r$-norm ($r\geq 2$) metric anisotropy is always weakly anisotropic.

For the $m$-fold anisotropic surface energy \eqref{kfold} with $\theta_0=0$, we have
\begin{equation}\label{kfoldaa1}
\gamma(\boldsymbol{p})=\left(p_1^2+p_2^2\right)^{\frac{1}{2}}(1+\beta \cos(m\theta)), \forall \boldsymbol{p}=(p_1,p_2)^T=|p|(-\sin\theta,\cos\theta)^T\in \mathbb{R}^2_*.
\end{equation}
Plugging \eqref{kfoldaa1} into \eqref{The xi-vector}, we get
\begin{align}
&\boldsymbol{\xi}=\boldsymbol{\xi}(\boldsymbol{n})=\boldsymbol{n}+\beta \cos(m\theta) \boldsymbol{n}+\beta m\sin(m\theta)\boldsymbol{n}^{\perp},
\quad \forall\boldsymbol{n}=(-\sin\theta,\cos\theta)\in \mathbb{S}^1,\\
&\lambda({\boldsymbol{n}})=1-\beta(m^2-1)\cos(m\theta),
\end{align}
which indicates that it is weakly anisotropic if $0<\beta\le \frac{1}{m^2-1}$; otherwise, it is strongly anisotropic.

For all the above $\gamma(\boldsymbol{n})$, their Hessian matrices are of the form:
\begin{equation}\label{Hessian for gamma}
\bf{H}_\gamma(\boldsymbol{n})=\lambda({\boldsymbol{n}})
\begin{pmatrix}n_2^2&-n_1n_2\\-n_1n_2&n_1^2\end{pmatrix}, \qquad \forall\boldsymbol{n}=(n_1,n_2)^T\in \mathbb{S}^1.
\end{equation}

\begin{center}
{\textbf{Appendix B.}} Proof of Lemma \ref{lem: Riemannian} for the Riemannian-like metric anisotropy
\end{center}
\setcounter{equation}{0}
\renewcommand{\theequation}{B.\arabic{equation}}

\medskip
\begin{proof}
First we consider the case $L=1$ and assume $\boldsymbol{G}_1=\left(\begin{smallmatrix}a&b\\b&c\end{smallmatrix}\right)
:=\boldsymbol{G}$ with $a>0$ and $ac-b^2>0$, then the minimal stabilizing function $k_0(\boldsymbol{n})$ becomes
\begin{equation}
k_0(\boldsymbol{n})=\gamma(\boldsymbol{n})^{-1}\text{Tr}(\boldsymbol{G})
=\gamma(\boldsymbol{n})^{-1}(a+c):=k_1(\boldsymbol{n}).
\end{equation}
By using $\boldsymbol{\xi}$ in \eqref{ellipsoidalaa1}, the corresponding surface energy matrix with respect to $k_1(\boldsymbol{n})$ can be given as
\begin{align}\label{AP: auxi 1}
\boldsymbol{Z}_{k_1}(\boldsymbol{n})&=
\gamma(\boldsymbol{n})I_2-\boldsymbol{\xi}\boldsymbol{n}^T
-\boldsymbol{n}\boldsymbol{\xi}^T+k_1(\boldsymbol{n})
\boldsymbol{n}\boldsymbol{n}^T\nonumber\\ &=\gamma(\boldsymbol{n})I_2-\gamma(\boldsymbol{n})^{-1}
\boldsymbol{G}\boldsymbol{n}\boldsymbol{n}^T-\gamma(\boldsymbol{n})^{-1}
\boldsymbol{n}\boldsymbol{n}^T\boldsymbol{G}+\gamma(\boldsymbol{n})^{-1}
(a+c)\boldsymbol{n}\boldsymbol{n}^T\nonumber\\ &=\gamma(\boldsymbol{n})^{-1}\begin{pmatrix}\gamma(\boldsymbol{n})^2
-2(an_1^2+bn_1n_2)+(a+c)n_1^2&*\\-(an_1n_2+bn_2^2)-(bn_1^2+cn_1n_2)
+(a+c)n_1n_2&*
\end{pmatrix}\nonumber\\
&=\gamma(\boldsymbol{n})^{-1}\begin{pmatrix}c&-b\\-b&a
\end{pmatrix}=\gamma(\boldsymbol{n})^{-1}\boldsymbol{J}^T
\boldsymbol{G}\boldsymbol{J},
\end{align}
where the $*$ means the entry can be deduced in the same way. By direct computations, we obtain
\begin{align*}
\gamma(\boldsymbol{n})\, (\hat{\boldsymbol{n}}^{\perp})^T\boldsymbol{Z}_{k_1}(\boldsymbol{n})
\hat{\boldsymbol{n}}^{\perp}-\gamma(\hat{
\boldsymbol{n}})^2&=(\hat{\boldsymbol{n}}^{\perp})^T
\boldsymbol{J}^T\boldsymbol{G}\boldsymbol{J}
\hat{\boldsymbol{n}}^{\perp}-\gamma(\hat{
\boldsymbol{n}})^2\\
&=\hat{\boldsymbol{n}}^T\boldsymbol{G}\hat{\boldsymbol{n}}-
\gamma(\boldsymbol{\hat{n}})^2=0.
\end{align*}
From the alternative definition of $k_0(\boldsymbol{n})$ in \eqref{alternative definition of k_0},
we obtain $k_0(\boldsymbol{n})\leq k_1(\boldsymbol{n})$ 

On the other hand, we take $\hat{\boldsymbol{n}}\to\boldsymbol{n}$ in $F(\boldsymbol{n}, \hat{\boldsymbol{n}})$. By applying \eqref{limFnnh} and the Hessian matrix derived in \eqref{ellipsoidalaa1} and \eqref{Hessian for gamma}, we then have
\begin{align*}
&(\boldsymbol{n}^\perp)^T\bf{H}_\gamma \boldsymbol{n}^{\perp}+\frac{|\boldsymbol{\xi}|^2}{\gamma(\boldsymbol{n})}\\ &=\gamma(\boldsymbol{n})^{-3}\left((ac-b^2)(n_2^4+2n_1^2n_2^2+n_1^4)
+(an_1+bn_2)^2+(bn_1+cn_2)^2\right)\\
&=\gamma(\boldsymbol{n})^{-3}(ac+a^2n_1^2+2abn_1n_2+2acn_1n_2+c^2n_2^2)\\
&=\gamma(\boldsymbol{n})^{-3}(an_1^2+2bn_1n_2+cn_2^2)(a+c)\\
&=\gamma(\boldsymbol{n})^{-1}(a+c)=k_0(\boldsymbol{n}),
\end{align*}
which means $k_0(\boldsymbol{n})\geq \gamma(\boldsymbol{n})^{-1}(a+c)$ by \eqref{existence of k_0(n)}, hence $k_0(\boldsymbol{n})=\gamma(\boldsymbol{n})^{-1}
\text{Tr}(\boldsymbol{G})=k_1(\boldsymbol{n})$.

For $L> 1$, Lemma \ref{lemma: sub-linear} yields $k_1(\boldsymbol{n})=\sum\limits_{l=1}^L \gamma_l(\boldsymbol{n})^{-1}\text{Tr}(\boldsymbol{G}_l)\geq k_0(\boldsymbol{n})$, and $\boldsymbol{Z}_{k_1}(\boldsymbol{n})$ can be derived by the same argument in \eqref{AP: auxi 1}.
\end{proof}

\color{black}
\begin{center}
{\textbf{Appendix C.}} Proof of Lemma \ref{lem: lr} for the $l^r$-norm  metric anisotropy
\end{center}
\setcounter{equation}{0}
\renewcommand{\theequation}{C.\arabic{equation}}

\medskip
\begin{proof}
(i) When $r=4$, a direct computation shows
\begin{align*}
&\gamma(\boldsymbol{n})\, (\hat{\boldsymbol{n}}^{\perp})^T\boldsymbol{Z}_{k_0}(\boldsymbol{n})
\hat{\boldsymbol{n}}^{\perp}-\gamma(\hat{
\boldsymbol{n}})^2\\
&=\frac{1-2n_1n_2\hat{n}_1\hat{n}_2}{\sqrt{n_1^4+n_2^4}}
-\sqrt{\hat{n}_1^4+\hat{n}_2^4}\\
&=\frac{(n_1^2+n_2^2)^2+(\hat{n}_1^2+\hat{n}_2^2)^2-4n_1n_2\hat{n}_1\hat{n}_2-2\sqrt{\hat{n}_1^4+\hat{n}_2^4}\sqrt{n_1^4+n_2^4}}{2\sqrt{n_1^4+n_2^4}}\\
&\geq\frac{(n_1^2+n_2^2)^2+(\hat{n}_1^2+\hat{n}_2^2)^2-4n_1n_2\hat{n}_1\hat{n}_2-n_1^4-n_2^4
-\hat{n}_1^4-\hat{n}_2^4}{2\sqrt{n_1^4+n_2^4}}\\
&=\frac{(n_1n_2-\hat{n}_1\hat{n}_2)^2}{\sqrt{n_1^4+n_2^4}}\geq 0,\qquad \forall
\boldsymbol{n}, \hat{\boldsymbol{n}}\in \mathbb{S}^1.
\end{align*}
By Theorem \ref{energy dissipation condition and k_0}, we get $k_0(\boldsymbol{n})\leq 2\gamma(\boldsymbol{n})^{-3}$.
On the other hand, by taking $\hat{\boldsymbol{n}}=(n_2, n_1)^T\in \mathbb{S}^1$ in \eqref{Fnnhat} and the $\boldsymbol{\xi}$ vector given in \eqref{eqxx}, we obtain
\begin{align*}
  F(\boldsymbol{n}, \hat{\boldsymbol{n}})&=\frac{2\gamma(\boldsymbol{n})(\gamma(\boldsymbol{n})^{-3}(n_1^3,n_2^3)\cdot(-n_1, n_2))(-n_1^2+n_2^2)}{\gamma(\boldsymbol{n})(-n_1^2+n_2^2)^2}\\
  &=2\gamma(\boldsymbol{n})^{-3}\frac{(-n_1^2+n_2^2)^2(n_2^2+n_1^2)}{(-n_1^2+n_2^2)^2}=2\gamma(\boldsymbol{n})^{-3}.
\end{align*}

By \eqref{existence of k_0(n)}, we know that $k_0(\boldsymbol{n})\geq 2\gamma(\boldsymbol{n})^{-3}$, hence $k_0(\boldsymbol{n})=2\gamma(\boldsymbol{n})^{-3}$.

(ii) When $r=6$, a direct computation shows
\begin{align*}
&\gamma(\boldsymbol{n})\, (\hat{\boldsymbol{n}}^{\perp})^TZ_{k_0}(\boldsymbol{n})
\hat{\boldsymbol{n}}^{\perp}-\gamma(\hat{
\boldsymbol{n}})^2\\
&=\gamma(\boldsymbol{n})^{-4}\,(1-n_1^2n_2^2-2n_1n_2\hat{n}_1\hat{n}_2)-\sqrt[3]{\hat{n}_1^6+\hat{n}_2^6}\\
&=\gamma(\boldsymbol{n})^{-4}\left(1-n_1^2n_2^2-2n_1n_2\hat{n}_1\hat{n}_2-\sqrt[3]{(n_1^6+n_2^6)^2(\hat{n}_1^6+\hat{n}_2^6)}\right)\\
&\geq\gamma(\boldsymbol{n})^{-4}\left(\frac{2(n_1^2+n_2^2)^3+(\hat{n}_1^2+\hat{n}_2^2)^3}{3}-n_1^2n_2^2-2n_1n_2\hat{n}_1\hat{n}_2-\frac{2(n_1^6+n_2^6)+(\hat{n}_1^6+\hat{n}_2^6)}{3}\right)\\
&=\gamma(\boldsymbol{n})^{-4}\left(\frac{6n_1^4n_2^2+6n_1^2n_2^4+3\hat{n}_1^4\hat{n}_2^2+3\hat{n}_1^2\hat{n}_2^4}{3}-n_1^2n_2^2-2n_1n_2\hat{n}_1\hat{n}_2\right)\\
&=\gamma(\boldsymbol{n})^{-4}\left(2n_1^2n_2^2(n_1^2+n_2^2)+\hat{n}_1^2\hat{n}_2^2(\hat{n}_1^2+\hat{n}_2^2)-n_1^2n_2^2-2n_1n_2\hat{n}_1\hat{n}_2\right)\\
&=\gamma(\boldsymbol{n})^{-4}(n_1n_2-\hat{n}_1\hat{n}_2)^2\geq 0,\qquad \forall
\boldsymbol{n}, \hat{\boldsymbol{n}}\in \mathbb{S}^1.
\end{align*}
By Theorem \ref{energy dissipation condition and k_0}, we get $k_0(\boldsymbol{n})\leq 2\gamma(\boldsymbol{n})^{-5}(n_2^4+n_2^2n_1^2+n_1^4)$.
On the other hand, by taking $\hat{\boldsymbol{n}}=(n_2, n_1)^T\in \mathbb{S}^1$ in \eqref{Fnnhat} and the $\boldsymbol{\xi}$ vector given in \eqref{eqxx}, we obtain
\begin{align*}
  F(\boldsymbol{n}, \hat{\boldsymbol{n}})&=\frac{2\gamma(\boldsymbol{n})(\gamma(\boldsymbol{n})^{-5}(n_1^5,n_2^5)\cdot(-n_1, n_2))(-n_1^2+n_2^2)}{\gamma(\boldsymbol{n})(-n_1^2+n_2^2)^2}\\
  &=2\gamma(\boldsymbol{n})^{-5}\frac{(-n_1^2+n_2^2)^2(n_2^4+n_2^2n_1^2+n_1^4)}{(-n_1^2+n_2^2)^2}=2\gamma(\boldsymbol{n})^{-5}(n_2^4+n_2^2n_1^2+n_1^4).
\end{align*}

By \eqref{existence of k_0(n)}, we know that $k_0(\boldsymbol{n})\geq 2\gamma(\boldsymbol{n})^{-5}(n_2^4+n_2^2n_1^2+n_1^4)$, hence $k_0(\boldsymbol{n})=2\gamma(\boldsymbol{n})^{-5}(n_2^4+n_2^2n_1^2+n_1^4)$.
\end{proof}

\medskip

\begin{center}
{\textbf{Appendix D.}} Proof of Lemma \ref{lem: Riemannian} for the $2/4$-fold  anisotropy
\end{center}
\setcounter{equation}{0}
\renewcommand{\theequation}{D.\arabic{equation}}

\medskip

\begin{proof}
For the $m$-fold anisotropy $\hat{\gamma}(\theta)=\gamma(\boldsymbol{n})=1+\beta \cos m\theta$, we know that $\hat{\gamma}'(\theta)=-m\beta \sin m\theta$. The $\tilde{F}^\theta(\hat{\theta})$ given in \eqref{tildeF} is
\begin{align}\label{tildeFD}
  \tilde{F}^\theta(\hat{\theta})=&2(1+\beta\cos m\theta)\nonumber\\
  &+\frac{(1+\beta\cos m\hat{\theta})^2-(1+\beta\cos m\theta)^2}{(1+\beta \cos m\theta)\sin^2(\hat{\theta}-\theta)}+
  \frac{m\beta\sin m\theta\;\sin (2(\hat{\theta}-\theta))}{\sin^2(\hat{\theta}-\theta)}.
\end{align}

(i) For the $2$-fold anisotropy, i.e. $m=2$, 
by applying Mathematica to \eqref{tildeFD}, we get
\begin{equation}\label{F_temp}
  \tilde{F}^\theta(\hat{\theta})=4-2(1+\beta\cos 2\theta)+\frac{2\beta^2(1-\cos 2(\hat{\theta}+\theta))}{1+\beta\cos 2\theta}.
\end{equation}
Thus by \eqref{opti} in Remark \ref{rmk 4.1}, we obtain
\begin{equation}\label{D3}
  k_0(\boldsymbol{n})=\max_{\hat{\theta}\in [\theta-\frac{\pi}{2}, \theta+\frac{\pi}{2}]}\tilde{F}^\theta(\hat{\theta})\leq 4-2\gamma(\boldsymbol{n})+\frac{4\beta^2}{\gamma(\boldsymbol{n})}.
\end{equation}
On the other hand, by taking $\hat{\theta}=\frac{\pi}{2}-\theta$ in \eqref{F_temp}, we obtain
\begin{equation}\label{D4}
  \tilde{F}^\theta(\frac{\pi}{2}-\theta)=4-2\gamma(\boldsymbol{n})+\frac{4\beta^2}{\gamma(\boldsymbol{n})}\leq k_0(\boldsymbol{n}).
\end{equation}
By combining \eqref{D3} and \eqref{D4}, we know $k_0(\boldsymbol{n})=4-2\gamma(\boldsymbol{n})+\frac{4\beta^2}{\gamma(\boldsymbol{n})}$, which validates \eqref{k02fold1}.

(ii) For the $4$-fold anisotropy, i.e. $m=4$, 
by applying Mathematica to \eqref{tildeFD}, we get
\begin{align}
  \tilde{F}^\theta(\hat{\theta})=
  &2\gamma(\boldsymbol{n})-\frac{16\beta\cos(\hat{\theta}-\theta)
  \cos(\hat{\theta}+3\theta)}{\gamma(\boldsymbol{n})}\nonumber\\
    &\quad -\frac{4\beta^2\cos(\hat{\theta}-\theta)(2\cos(\hat{\theta}+
    7\theta)+\cos(3\hat{\theta}+5\theta)+\cos(5\hat{\theta}+3\theta))}
    {\gamma(\boldsymbol{n})}.
\end{align}
Thus by \eqref{opti} in Remark \ref{rmk 4.1}, we obtain
\begin{equation}
  k_0(\boldsymbol{n})=\max_{\hat{\theta}\in [\theta-\frac{\pi}{2}, \theta+\frac{\pi}{2}]}\tilde{F}^\theta(\hat{\theta})\leq 2\gamma(\boldsymbol{n})+\frac{16\beta+16\beta^2}{\gamma(\boldsymbol{n})}=k_1(\boldsymbol{n}),
\end{equation}
which validates \eqref{k02fold2}.
\end{proof}
\color{black}





\begin{thebibliography}{10}

\bibitem{bansch2004surface}
{\sc E.~B{\"a}nsch, P.~Morin, and R.~H. Nochetto}, {\em Surface diffusion of
  graphs: variational formulation, error analysis, and simulation}, SIAM J.
  Numer. Anal., 42 (2004), pp.~773--799.

\bibitem{bao2022volume}
{\sc W.~Bao, H.~Garcke, R.~N{\"u}rnberg, and Q.~Zhao}, {\em Volume-preserving
  parametric finite element methods for axisymmetric geometric evolution
  equations}, J. Comput. Phys., 460 (2022), p.~111180.

\bibitem{Bao17}
{\sc W.~Bao, W.~Jiang, D.~J. Srolovitz, and Y.~Wang}, {\em Stable equilibria of
  anisotropic particles on substrates: a generalized winterbottom
  construction}, SIAM J. Appl. Math., 77 (2017), pp.~2093--2118.

\bibitem{bao2017parametric}
{\sc W.~Bao, W.~Jiang, Y.~Wang, and Q.~Zhao}, {\em A parametric finite element
  method for solid-state dewetting problems with anisotropic surface energies},
  J. Comput. Phys., 330 (2017), pp.~380--400.

\bibitem{bao2021structurepreserving}
{\sc W.~Bao and Q.~Zhao}, {\em A structure-preserving parametric finite element
  method for surface diffusion}, SIAM J. Numer. Anal., 59 (2021),
  pp.~2775--2799.

\bibitem{barrett2007parametric}
{\sc J.~W. Barrett, H.~Garcke, and R.~N{\"u}rnberg}, {\em A parametric finite
  element method for fourth order geometric evolution equations}, J. Comput.
  Phys., 222 (2007), pp.~441--467.

\bibitem{barrett2008numerical}
{\sc J.~W. Barrett, H.~Garcke, and R.~N{\"u}rnberg}, {\em Numerical
  approximation of anisotropic geometric evolution equations in the plane}, IMA
  J. Numer. Anal., 28 (2008), pp.~292--330.

\bibitem{barrett2008variational}
{\sc J.~W. Barrett, H.~Garcke, and R.~N{\"u}rnberg}, {\em A variational
  formulation of anisotropic geometric evolution equations in higher
  dimensions}, Numer. Math., 109 (2008), pp.~1--44.

\bibitem{Barrett11}
{\sc J.~W. Barrett, H.~Garcke, and R.~N{\"u}rnberg}, {\em The approximation of
  planar curve evolutions by stable fully implicit finite element schemes that
  equidistribute}, Numer. Methods Part. Differ. Equa., 27 (2011), pp.~1--30.

\bibitem{Barrett2020}
{\sc J.~W. Barrett, H.~Garcke, and R.~N{\"u}rnberg}, {\em Parametric finite
  element approximations of curvature-driven interface evolutions}, in Handb.
  Numer. Anal., vol.~21, Elsevier, 2020, pp.~275--423.

\bibitem{cahn1991stability}
{\sc J.~Cahn}, {\em Stability, microstructural evolution, grain growth, and
  coarsening in a two-dimensional two-phase microstructure}, Acta Metall.
  Mater., 39 (1991), pp.~2189--2199.

\bibitem{Cahn94}
{\sc J.~W. Cahn and J.~E. Taylor}, {\em Surface motion by surface diffusion},
  Acta Metall. Mater., 42 (1994), pp.~1045--1063.

\bibitem{clarenz2000anisotropic}
{\sc U.~Clarenz, U.~Diewald, and M.~Rumpf}, {\em Anisotropic geometric
  diffusion in surface processing}, IEEE Visualization 2000, 2000.

\bibitem{deckelnick2005computation}
{\sc K.~Deckelnick, G.~Dziuk, and C.~M. Elliott}, {\em Computation of geometric
  partial differential equations and mean curvature flow}, Acta Numer., 14
  (2005), pp.~139--232.

\bibitem{deckelnick2005fully}
{\sc K.~Deckelnick, G.~Dziuk, and C.~M. Elliott}, {\em Fully discrete finite
  element approximation for anisotropic surface diffusion of graphs}, SIAM J.
  Numer. Anal., 43 (2005), pp.~1112--1138.

\bibitem{du2010tangent}
{\sc P.~Du, M.~Khenner, and H.~Wong}, {\em A tangent-plane marker-particle
  method for the computation of three-dimensional solid surfaces evolving by
  surface diffusion on a substrate}, J. Comput. Phys., 229 (2010),
  pp.~813--827.

\bibitem{Fonseca14}
{\sc I.~Fonseca, A.~Pratelli, and B.~Zwicknagl}, {\em Shapes of epitaxially
  grown quantum dots}, Arch. Ration. Mech. Anal., 214 (2014), pp.~359--401.

\bibitem{Hausser05}
{\sc F.~Hau{\ss}er and A.~Voigt}, {\em A discrete scheme for regularized
  anisotropic surface diffusion: a 6th order geometric evolution equation},
  Interfaces Free Bound., 7 (2005), pp.~353--370.

\bibitem{Hausser07}
{\sc F.~Hau{\ss}er and A.~Voigt}, {\em A discrete scheme for parametric
  anisotropic surface diffusion}, J. Sci. Comput., 30 (2007), pp.~223--235.

\bibitem{hoffman1972vector}
{\sc D.~W. Hoffman and J.~W. Cahn}, {\em A vector thermodynamics for
  anisotropic surfaces: I. fundamentals and application to plane surface
  junctions}, Surface Science, 31 (1972), pp.~368--388.

\bibitem{huang2021theta}
{\sc W.~Huang, W.~Jiang, and Q.~Zhao}, {\em A {$\theta-L$} formulation-based
  finite element method for solving axisymmetric solid-state dewetting
  problems}, East Asian J. Appl. Math., 11 (2021), pp.~389--405.

\bibitem{Jiang2012}
{\sc W.~Jiang, W.~Bao, C.~V. Thompson, and D.~J. Srolovitz}, {\em Phase field
  approach for simulating solid-state dewetting problems}, Acta Mater., 60
  (2012), pp.~5578--5592.

\bibitem{jiang2021}
{\sc W.~Jiang and B.~Li}, {\em A perimeter-decreasing and area-conserving
  algorithm for surface diffusion flow of curves}, J. Comput. Phys., 443
  (2021), p.~110531.

\bibitem{Jiang2016solid}
{\sc W.~Jiang, Y.~Wang, Q.~Zhao, D.~J. Srolovitz, and W.~Bao}, {\em Solid-state
  dewetting and island morphologies in strongly anisotropic materials}, Scr.
  Mater., 115 (2016), pp.~123--127.

\bibitem{jiang2019sharp}
{\sc W.~Jiang and Q.~Zhao}, {\em Sharp-interface approach for simulating
  solid-state dewetting in two dimensions: A {Cahn--Hoffman} $\xi$-vector
  formulation}, Phys. D, 390 (2019), pp.~69--83.

\bibitem{Jiang}
{\sc W.~Jiang, Q.~Zhao, and W.~Bao}, {\em Sharp-interface model for simulating
  solid-state dewetting in three dimensions}, SIAM J. Appl. Math., 80 (2020),
  pp.~1654--1677.

\bibitem{Kovacs21}
{\sc B.~Kov{\'a}cs, B.~Li, and C.~Lubich}, {\em A convergent evolving finite
  element algorithm for willmore flow of closed surfaces}, Numer. Math., 149
  (2021), pp.~595--643.

\bibitem{Li2009geo}
{\sc B.~Li, J.~Lowengrub, A.~Ratz, and A.~Voigt}, {\em Geometric evolution laws
  for thin crystalline films: modeling and numerics}, Commun. Comput. Phys., 6
  (2009), p.~433.

\bibitem{li2020energy}
{\sc Y.~Li and W.~Bao}, {\em An energy-stable parametric finite element method
  for anisotropic surface diffusion}, J. Comput. Phys., 446 (2021), p.~110658.

\bibitem{li1999numerical}
{\sc Z.~Li, H.~Zhao, and H.~Gao}, {\em A numerical study of electro-migration
  voiding by evolving level set functions on a fixed {Cartesian} grid}, J.
  Comput. Phys., 152 (1999), pp.~281--304.

\bibitem{Mayer01}
{\sc U.~F. Mayer}, {\em Numerical solutions for the surface diffusion flow in
  three space dimensions}, Comput. Appl. Math., 20 (2001), pp.~361--379.

\bibitem{Mullins57}
{\sc W.~W. Mullins}, {\em Theory of thermal grooving}, J. Appl. Phys., 28
  (1957), pp.~333--339.

\bibitem{Spencer04}
{\sc B.~J. Spencer}, {\em Asymptotic solutions for the equilibrium crystal
  shape with small corner energy regularization}, Phys. Rev. E, 69 (2004),
  p.~011603.

\bibitem{Sutton95}
{\sc A.~P. Sutton and R.~W. Balluffi}, {\em Interfaces in Crystalline
  Materials}, Clarendon Press, 1995.

\bibitem{taylor1992ii}
{\sc J.~E. Taylor}, {\em Mean curvature and weighted mean curvature}, Acta
  Metall. Mater., 40 (1992), pp.~1475--1485.

\bibitem{taylor1994linking}
{\sc J.~E. Taylor and J.~W. Cahn}, {\em Linking anisotropic sharp and diffuse
  surface motion laws via gradient flows}, J. Stat. Phys., 77 (1994),
  pp.~183--197.

\bibitem{Thompson12}
{\sc C.~V. Thompson}, {\em Solid-state dewetting of thin films}, Annu. Rev.
  Mater. Res., 42 (2012), pp.~399--434.

\bibitem{Torabi09}
{\sc S.~Torabi, J.~Lowengrub, A.~Voigt, and S.~Wise}, {\em A new phase-field
  model for strongly anisotropic systems}, Proc. R. Soc. A: Math. Phys. Eng.
  Sci., 465 (2009), pp.~1337--1359.

\bibitem{wang2015sharp}
{\sc Y.~Wang, W.~Jiang, W.~Bao, and D.~J. Srolovitz}, {\em Sharp interface
  model for solid-state dewetting problems with weakly anisotropic surface
  energies}, Phys. Rev. B, 91 (2015), p.~045303.

\bibitem{wheeler1999cahn}
{\sc A.~Wheeler}, {\em {Cahn--Hoffman} $\xi$-vector and its relation to diffuse
  interface models of phase transitions}, J. Stat. Phys., 95 (1999),
  pp.~1245--1280.

\bibitem{wong2000periodic}
{\sc H.~Wong, P.~Voorhees, M.~Miksis, and S.~Davis}, {\em Periodic mass
  shedding of a retracting solid film step}, Acta Mater., 48 (2000),
  pp.~1719--1728.

\bibitem{Wulff01}
{\sc G.~Wulff}, {\em Zur frage der geschwindigkeit des wachstums und der
  aufl\"{o}sung der krystallfl\"{a}chen}, Z. Kristallogr, 34 (1901),
  pp.~449--530.

\bibitem{Suo97}
{\sc L.~Xia, A.~F. Bower, Z.~Suo, and C.~Shih}, {\em A finite element analysis
  of the motion and evolution of voids due to strain and electromigration
  induced surface diffusion}, J. Mech. Phys. Solids, 45 (1997), pp.~1473--1493.

\bibitem{xu2009local}
{\sc Y.~Xu and C.-W. Shu}, {\em Local discontinuous {Galerkin} method for
  surface diffusion and {Willmore flow} of graphs}, J. Sci. Comput., 40 (2009),
  pp.~375--390.

\bibitem{Ye10a}
{\sc J.~Ye and C.~V. Thompson}, {\em Mechanisms of complex morphological
  evolution during solid-state dewetting of single-crystal nickel thin films},
  Appl. Phys. Lett., 97 (2010), p.~071904.

\bibitem{Zhao}
{\sc Q.~Zhao, W.~Jiang, and W.~Bao}, {\em A parametric finite element method
  for solid-state dewetting problems in three dimensions}, SIAM J. Sci.
  Comput., 42 (2020), pp.~B327--B352.

\bibitem{bao2020energy}
{\sc Q.~Zhao, W.~Jiang, and W.~Bao}, {\em An energy-stable parametric finite
  element method for simulating solid-state dewetting}, IMA J. Numer. Anal., 41
  (2021), pp.~2026--2055.

\end{thebibliography}
\end{document}